\documentclass[12pt]{amsart}
\usepackage[margin=1in]{geometry}

\usepackage{amssymb}
\usepackage{enumerate,url,amssymb,amsmath,amsthm,amsxtra,mathtools,mathrsfs,calc,nccmath,color,bm}
\usepackage{calc}
\usepackage{graphicx, caption, subcaption}
\usepackage{needspace}
\usepackage{hyperref}
\usepackage{mathtools}

\newcommand{\Z}{\mathbb{Z}}
\newcommand{\Q}{\mathbb{Q}}
\newcommand{\R}{\mathbb{R}}

\newcommand{\N}{\mathbb{N}}
\newcommand{\C}{\mathbb{C}}

\newcommand{\GL}{\operatorname{GL}}

\renewcommand{\epsilon}{\varepsilon}

\newcommand{\re}{\mathrm{Re}}
\newcommand{\im}{\mathrm{Im}}

\DeclareFontFamily{U}{wncy}{}
\DeclareFontShape{U}{wncy}{m}{n}{<->wncyr10}{}
\DeclareSymbolFont{mcy}{U}{wncy}{m}{n}

\newtheorem{theorem}{Theorem}[section]
\newtheorem{lemma}[theorem]{Lemma}
\newtheorem{proposition}[theorem]{Proposition}

\newtheorem{corollary}[theorem]{Corollary}

\theoremstyle{definition}

\newtheorem*{theorem*}{Theorem}

\theoremstyle{remark}
\newtheorem*{remark}{Remark}

\numberwithin{equation}{section}

\renewcommand{\tilde}{\widetilde}
\renewcommand{\epsilon}{\varepsilon}

\newcommand{\T}{\mathbb{T}}
\renewcommand{\C}{\mathbb{C}}

\renewcommand{\Re}{\operatorname{Re}}

\newcommand{\pfrac}[2]{\Big(\frac{#1}{#2}\Big)}

\theoremstyle{definition}

\begin{document}

\title[Distribution of imaginary parts of $L$-functions]{A zero density estimate and fractional imaginary parts of zeros for $\GL_2$ $L$-functions}

\address{Department of Mathematics, University of Illinois at Urbana-Champaign,
1409 West Green Street, Urbana, IL 61801, USA}

\author{Olivia Beckwith}
\email{obeck@illinois.edu}

\author{Di Liu}
\email{dil4@illinois.edu}

\author{Jesse Thorner}
\email{jesse.thorner@gmail.com}

\author{Alexandru Zaharescu}
\address{Department of Mathematics, University of Illinois at Urbana-Champaign,
1409 West Green Street, Urbana, IL 61801, USA, and Simon Stoilow Institute of Mathematics of the Romanian Academy,
P.O. Box 1-764, RO-014700 Bucharest, Romania}
\email{zaharesc@illinois.edu}

\subjclass[2020]{11F41, 11M41}

\keywords{$L$-functions; zero density estimate; Sato-Tate conjecture; fractional parts}

\date{}

\begin{abstract}
    We prove an analogue of Selberg's zero density estimate for $\zeta(s)$ that holds for any $\GL_2$ $L$-function.  We use this estimate to study the distribution of the vector of fractional parts of $\gamma\bm{\alpha}$, where $\bm{\alpha}\in\R^n$ is fixed and $\gamma$ varies over the imaginary parts of the nontrivial zeros of a $\GL_2$ $L$-function.
\end{abstract}

\maketitle

\section{Introduction and statement of results}

\subsection{Main results}

Let $\zeta(s)$ be the Riemann zeta function, and let $N(\sigma,T)$ denote the number of zeros
$\beta+i\gamma$ of $\zeta(s)$ with $\beta\geq\sigma\geq 0$ and $|\gamma|\leq T$. The asymptotic
\begin{equation}
    \label{eqn:riemann_von_mangoldt}
    N(T):= N(0,T)\sim (1/\pi)T\log T
\end{equation}
follows from the argument principle. The Riemann hypothesis asserts that $\zeta(s)\neq 0$ for $\re(s)>\frac{1}{2}$, so $N(\sigma,T)=0$ for $\sigma>\frac{1}{2}$. Selberg \cite[Thm 9.19C]{titchmarsh2} proved a delicate zero density estimate  that recovers the upper bound in  \eqref{eqn:riemann_von_mangoldt} at $\sigma=\frac{1}{2}$, namely
\begin{equation}
\label{eqn:selberg_zde}
    N(\sigma,T)\ll T^{1-\frac{1}{4}(\sigma-\frac{1}{2})}\log T,\qquad \sigma\geq \tfrac{1}{2}.
\end{equation}
(See Baluyot \cite[Theorem 1.2.1]{Baluyot} for an improvement.)  Selberg's estimate implies that
\[
\frac{1}{N(T)}\#\Big\{\rho=\beta+i\gamma\colon \zeta(\rho)=0,~|\gamma|\leq T,~\beta\in\Big[\frac{1}{2}-\frac{4\log\log T}{\log T},\frac{1}{2} +\frac{4\log\log T}{\log T}\Big]\Big\}\ll \frac{1}{\log T}.
\]
As an application of \eqref{eqn:selberg_zde}, Selberg proved a central limit theorem for $\log|\zeta(\frac{1}{2}+it)|$:
\begin{equation}
\label{eqn:Selberg_CLT}
\lim_{T\to\infty}\frac{1}{T}\textup{meas}\Big\{T\leq t\leq 2T\colon \log|\zeta(\tfrac{1}{2}+it)|
    \geq V\sqrt{\tfrac{1}{2}\log\log T}\Big\}\sim\frac{1}{\sqrt{2\pi}}\int_V^{\infty}e^{-u^2/2}du.
\end{equation}

Let $\mathscr{A}$ be the set of cuspidal automorphic representations of $\GL_2$ over $\Q$ with unitary central character.  For $\pi\in\mathscr{A}$, let $L(s,\pi)$ be its standard $L$-function.  Define
\[
    N_{\pi}(\sigma,T):=\#\{\rho=\beta+i\gamma\colon \beta\geq\sigma,~|\gamma|\leq T,~L(\rho,\pi)=0\}.
\]
As with $\zeta(s)$, the argument principle can be used to prove that
\begin{equation}
    \label{eq:number_of_zeros}
    N_{\pi}(T):= N_{\pi}(0,T)\sim (2/\pi)T\log T.
\end{equation}
The generalized Riemann hypothesis (GRH) asserts that $L(s, \pi) \neq 0$ for $\re(s) > \frac{1}{2}$.
When $\pi \in \mathscr{A}$ corresponds with a holomorphic cuspidal newform of even weight $k \geq 2$, Luo \cite[Theorem 1.1]{Luo} and Li \cite[Section 7]{FZ2} proved that $N_{\pi}(\sigma, T) \ll T^{1-\frac{1}{72}(\sigma-\frac{1}{2})}\log T$.  We prove:

\begin{theorem}\label{thm:ZDE}
    Let $\theta\in[0,\frac{7}{64}]$ be an admissible exponent toward the generalized Ramanujan conjecture for Hecke--Maa\ss{} newforms (see \eqref{eqn:GRC}), and fix $0<c<\frac{1}{4}-\frac{\theta}{2}$. If $\pi \in \mathscr{A}$, $\sigma \geq \frac{1}{2}$, and $T \geq 2$, then $N_{\pi}(\sigma, T) \ll T^{1-c(\sigma-\frac{1}{2})}\log T$.  The implied constant depends at most on $\pi$.
\end{theorem}

\begin{remark}
    Under the generalized Ramanujan conjecture for all Hecke--Maa\ss{} newforms, we may take $\theta=0$.  In this case, our result is as strong as Selberg's zero density estimate \eqref{eqn:selberg_zde} for $\zeta(s)$.  Currently, the best unconditional bound is $\theta\leq \frac{7}{64}$, so we may choose any $c<\frac{25}{128}$.  This noticeably improves the work of Luo and Li, and it holds for any $\pi\in\mathscr{A}$.  The constant $\frac{1}{4}-\frac{\theta}{2}$ is, as of now, the supremum over all $\varpi$ for which we can unconditionally prove an asymptotic for the second mollified moment of $L(s,\pi)$ on $\re(s)=\frac{1}{2}+\frac{1}{\log T}$ with a mollifier of length $T^{\varpi}$.
\end{remark}


\begin{corollary}
    \label{cor:CLT}
    Let $\pi \in \mathscr{A}$.  If $V\in\R$, then as $T\to\infty$, we have
    \[
    \lim_{T\to\infty}\frac{1}{T}\textup{meas}\Big\{T\leq t\leq 2T\colon \log|L(\tfrac{1}{2}+it,\pi)|
        \geq V\sqrt{\tfrac{1}{2}\log\log T}\Big\}=\frac{1}{\sqrt{2\pi}}\int_V^{\infty}e^{-u^2/2}du,
    \]
    \[
    \lim_{T\to\infty}\frac{1}{T}\textup{meas}\Big\{T\leq t\leq 2T\colon \arg L(\tfrac{1}{2}+it,\pi)
        \geq V\sqrt{\tfrac{1}{2}\log\log T}\Big\}=\frac{1}{\sqrt{2\pi}}\int_V^{\infty}e^{-u^2/2}du.
    \]
\end{corollary}

\begin{proof}
    Bombieri and Hejhal \cite[Thm B]{BH} proved that this follows from a zero density estimate of the quality given by Theorem \ref{thm:ZDE}.
\end{proof}

\begin{remark}
    Radziwi{\l}{\l} and Soundararajan \cite{RS_Selberg} recently found a second proof of \eqref{eqn:Selberg_CLT}
    which avoids the use of zero density estimates. Their work was recently extended to holomorphic newforms by Das \cite{Das}.  The proof relies on both the generalized Ramanujan conjecture and the Sato--Tate conjecture,
    neither of which is known for any Hecke--Maa\ss{} newform.
\end{remark}

We study the distribution of imaginary parts of the nontrivial zeros of $L(s, \pi)$ using Theorem \ref{thm:ZDE}.
First, we give an incomplete history of such results for $\zeta(s)$. Hlawka \cite{Hlawka} proved
that if $\alpha\in\R$ is fixed and $h\colon \T\to\C$ is continuous, then
\begin{equation}
\label{eqn:Hlawka}
     \lim_{T\to\infty}\frac{1}{N(T)}\sum_{|\gamma|\leq T}h(\alpha\gamma)=\int_{\T}h(t)dt,
\end{equation}
where $\gamma$ varies over the imaginary parts of the nontrivial zeros of $\zeta(s)$ and $\T=\R/\Z$.  Thus the sequence of fractional parts $\{\gamma\alpha\}$ is equidistributed modulo 1.  However, given a rate of convergence, there exist continuous functions $h$ such that the limit in \eqref{eqn:Hlawka} cannot be attained with said rate (see also \cite[Thm 7]{FZ}).  Therefore, \eqref{eqn:Hlawka} is the best that one can say for arbitrary $h$.

Ford and Zaharescu \cite[Corollary 2]{FZ} established the existence of a second order term, proving that if $h:\T\to\mathbb{C}$ is twice continuously differentiable,\footnote{On RH, absolute continuity suffices.} then
\begin{equation}
    \label{eqn:FZ}
    \sum_{|\gamma|\leq T}h(\alpha\gamma)=N(T)\int_{\T}h(t)dt+T\int_{\T}h(t)g_{\alpha}(t)dt+o(T),
\end{equation}
where
\begin{equation}
\label{eqn:g_func_FZ}
g_{\alpha}(t) = \begin{dcases*}
\frac{\log p}{\pi}\re\sum_{k=1}^{\infty}\frac{e^{-2\pi i qkt}}{p^{ak/2}} & \parbox{.5\textwidth}{if there exists a prime $p$ and $a,q\in\Z$ such that $\gcd(a,q)=1$ and $\alpha=\frac{a}{q}\frac{\log p}{2\pi}$,}   \\
0 & otherwise.
\end{dcases*}
\end{equation}
Despite the limitations on the analytic properties of $h$, Ford and Zaharescu still conjectured \cite[Conjecture A]{FZ} that for any interval $\mathbb{I}\subseteq\T$ of length $|\mathbb{I}|$, we have
\begin{equation}
\label{eqn:conj1_zeta}
\sum_{\substack{|\gamma|\leq T \\ \{\alpha\gamma\}\in\mathbb{I}}}1 = |\mathbb{I}|N(T)+T\int_{\mathbb{I}}g_{\alpha}(t)dt+o(T),
\end{equation}
which implies that
\begin{equation}
\label{eqn:conj2_zeta}
D_{\alpha}(T):=\sup_{\mathbb{I}\subseteq \T}\Big|\frac{1}{N(T)}\sum_{\substack{|\gamma|\leq T \\ \{\alpha\gamma\}\in\mathbb{I}}}1 -|\mathbb{I}|\Big|=\frac{T}{N(T)}\sup_{\mathbb{I}\subseteq\T}\Big|\int_{\mathbb{I}}g_{\alpha}(t)dt\Big|+o\Big(\frac{1}{\log T}\Big).
\end{equation}

Ford, Soundararajan, and Zaharescu \cite{FSZ} made some progress toward the conjectured asymptotics \eqref{eqn:conj1_zeta} and \eqref{eqn:conj2_zeta}.  Unconditionally, they proved that
\begin{equation}
\label{eqn:thm_fsz_1.1}
D_{\alpha}(T)\geq \frac{T}{N(T)}\sup_{\mathbb{I}\subseteq\T}\Big|\int_{\mathbb{I}}g_{\alpha}(t)dt\Big|+o\Big(\frac{1}{\log T}\Big).
\end{equation}
Assuming RH, they proved that
\begin{equation}
\label{eqn:thm_fsz_1.2}
\Big|\sum_{\substack{|\gamma|\leq T \\ \{\alpha\gamma\}\in\mathbb{I}}}1 - |\mathbb{I}|N(T)-T\int_{\mathbb{I}}g_{\alpha}(t)dt\Big|\leq \Big(\frac{\alpha}{2}+o(1)\Big)T.
\end{equation}
Along with making some appealing connections between the conjectured asymptotics \eqref{eqn:conj1_zeta} and \eqref{eqn:conj2_zeta} and other intriguing open problems like pair correlation of zeros of $\zeta(s)$ and the distribution of primes in short intervals, they proved analogues for other $L$-functions of \eqref{eqn:thm_fsz_1.1} (assuming a zero density estimate of the form \eqref{eqn:selberg_zde}) and \eqref{eqn:thm_fsz_1.2} (assuming GRH).

In this paper, we extend the work in \cite{FZ,FSZ,FMZ,LZ} to $L(s,\pi)$ for any $\pi\in\mathscr{A}$ using Theorem \ref{thm:ZDE}.  Let $n \ge 1$.  Consider the $\bm{\alpha}\in\R^n$ for which there exists a constant $C_{\bm{\alpha}}>0$ such that\footnote{The set of vectors $\bm{\alpha}$ for which there exists $C_{\bm{\alpha}}>0$ such that \eqref{eq:condition_alpha}
holds have full Lebesgue measure in $\R^n$ by work of Kemble \cite{MR2261844} and Khintchine \cite{MR1512207}.}
\begin{equation}
    \label{eq:condition_alpha}
    |\bm{m} \cdot \bm{\alpha}| \geq C_{\bm{\alpha}}e^{-\|\bm{m}\|_2 }\qquad \textup{for all $\bm{m}\in\Z^n \setminus \{\bm{0}\}$},
\end{equation}
where $\|\bm{m}\|_p$ is the $\ell^p$ norm on $\R^n$ for $1\leq p\leq\infty$.  This is a technical artifact of our extension to $\R^n$; when $n=1$, the condition reduces to $\alpha\neq 0$.
Our density function $g_{\pi, \bm{\alpha}}(\bm{t})$, which extends \eqref{eqn:g_func_FZ} for $n\geq 2$,
is identically zero unless there exists a matrix $M=(b_{jk}) \in \mathcal{M}_{r\times n}(\Z)$
with linearly independent row vectors $\bm{b}_j$ and $\gcd(b_{j1}, \ldots, b_{jn}) = 1$ for all $1 \leq j \leq r$;
fully reduced rationals $a_1/q_1, \ldots, a_r/q_r$; and distinct primes $p_1, \ldots, p_r$ such that
\begin{equation}
    \label{eqn:matrix_condition}
    M\bm{\alpha}^{\intercal}=\Big(\frac{a_1}{q_1}\frac{\log p_1}{2\pi},\ldots,\frac{a_r}{q_r}\frac{\log p_r}{2\pi}\Big)^{\intercal}.
\end{equation}
Among such possible matrices $M$, choose one with maximal $r$, which uniquely determines the row vectors $\bm{b}_j=(b_{j1},\ldots,b_{jn})$.  If such an $M$ exists, then define
\begin{equation}
    \label{eq:densityfunction}
    g_{\pi, \bm{\alpha}}(\bm{t}) :=
    -\frac{2}{\pi} \Re\sum_{j = 1}^{r} \sum_{l = 1}^{\infty} \frac{\Lambda_{\pi}({p_j}^{a_jl})}{{p_j}^{a_jl/2}} e^{-2 \pi i q_j l (\bm{b_j} \cdot \bm{t})},
\end{equation}
where
\[
    -\frac{L^\prime(s,\pi)}{L(s,\pi)} = \sum_{n =1}^{\infty} \frac{\Lambda_{\pi}(n)}{n^s},\qquad \re(s)>1.
\]

Let $\pi\in\mathscr{A}$, $n\geq 1$, and $\mathbb{B} \subseteq \T^n$ be a product of $n$ subintervals of $\T$.  The conjecture in \eqref{eqn:conj1_zeta} can be extended to $\pi\in\mathscr{A}$ as follows:
\begin{equation}
\label{eqn:conj_GL2}
\sum_{\substack{|\gamma|\leq T \\ \{\gamma \bm{\alpha}\} \in \mathbb{B}}} 1 = \mathrm{vol}(\mathbb{B}) N_{\pi}(T) + T \int_{\mathbb{B}}g_{\pi, \bm{\alpha}}(\bm{t}) \textup{d}\bm{t} + o(T),
\end{equation}
where $\gamma$ ranges over the imaginary parts of the nontrivial zeros of $L(s,\pi)$.  As progress toward \eqref{eqn:conj_GL2}, we prove an unconditional $n$-dimensional version of \eqref{eqn:FZ} for $\GL_2$ $L$-functions.  In what follows, let $C^u(\T^n)$ be the set of $u$-times continuously differentiable functions $h\colon\T^n\to\mathbb{R}$.  Let $\gamma$ vary over the imaginary parts of the nontrivial zeros of $L(s, \pi)$.

\begin{theorem}\label{thm:densityfunction}
    Let $\pi \in \mathscr{A}$. Let $\bm{\alpha} \in \mathbb{R}^n$ satisfy \eqref{eq:condition_alpha}.
    If $h \in C^{n+2}(\T^n)$, then
    \[
        \sum_{|\gamma| \leq T} h(\gamma \bm{\alpha})
        = N_{\pi}(T)\int_{\mathbb{T}^n} h(\bm{t})\textup{d}\bm{t} + T \int_{\mathbb{T}^n} h(\bm{t}) g_{\pi, \bm{\alpha}}(\bm{t})\textup{d}\bm{t}+o(T),
    \]
    where $\textup{d}\bm{t}$ is Lebesgue measure on $\mathbb{T}^n$.
    The implied constant depends on $\pi$, $h$, and $\bm{\alpha}$.
\end{theorem}

\begin{remark}
    With extra work, we may allow $h\in C^{n+1}(\T^n)$.  Also, under some extra conditions on $\bm{\alpha}$ which excludes a density zero subset of $\mathbb{R}^n$, we expect to improve the error term to $O(T(\log\log T)/\log T)$.  We will return to this in future work.
\end{remark}

Define the discrepancy
\[
    D_{\pi, \bm{\alpha}}(T) := \sup_{\mathbb{B}\subseteq\T^n} \Big|\frac{1}{N_{\pi}(T)}
    \sum_{\substack{|\gamma|\leq T \\ \{\gamma \bm{\alpha}\} \in \mathbb{B}}}1-\mathrm{vol}(\mathbb{B})\Big|.
\]
The following result follows quickly from Theorem \ref{thm:densityfunction}.

\begin{corollary}
    \label{cor:FSZ}
    Let $\pi\in\mathscr{A}$.  If $\bm{\alpha}\in\R^n$ satisfies \eqref{eq:condition_alpha}, then
    \[
        D_{\pi, \bm{\alpha}}(T) \geq \frac{T}{N_{\pi}(T)}\int_{\T^n}g_{\pi, \alpha}(\bm{t})\textup{d}\bm{t}+o\Big(\frac{1}{\log T}\Big).
    \]
\end{corollary}
When $n=1$, we recover an unconditional analogue of \eqref{eqn:thm_fsz_1.1} for all $\pi\in\mathscr{A}$.This special case of Corollary \ref{cor:FSZ} was proved in \cite{FSZ} under the hypothesis of a zero density estimate like that in Theorem \ref{thm:ZDE}.

\begin{remark}
    Let $d\geq 3$ be an integer, and let $L(s,\pi)$ be the standard $L$-function associated to a cuspidal automorphic representation $\pi$ of $\GL_d$ over $\Q$.
    Let $N_{\pi}(\sigma,T)$ be the number of nontrivial zeros $\beta+i\gamma$ of $L(s,\pi)$ with $\beta\geq\sigma$ and $|\gamma|\leq T$.
    If there exist constants $c_{\pi}>0$ and $d_{\pi}>0$ such that $N_{\pi}(\sigma,T)\ll T^{1-c_{\pi}(\sigma-\frac{1}{2})}(\log T)^{d_{\pi}}$,
    then one can prove an analogue of Theorem \ref{thm:densityfunction} and Corollary \ref{cor:FSZ} for $L(s,\pi)$.
    Such an estimate for $N_{\pi}(\sigma,T)$ is not yet known for any $d\geq 3$, and it appears to be quite difficult to prove.
\end{remark}

\subsection{Application to ``zero races''}
In a letter to Fuss, Chebyshev observed that the generalized Riemann hypothesis implies that primes $p\equiv 3\pmod{4}$ tend to be more numerous than primes $p\equiv 1\pmod{4}$.  Using the generalized Riemann hypothesis and other hypotheses, Rubinstein and Sarnak \cite{RS} began a systematic study of ``prime number races'' in which they determine how often $\pi(x;4,3)>\pi(x;4,1)$, where $\pi(x;q,a)$ equals $\#\{p\leq x\colon p\equiv a\pmod{q}\}$.  They proved that
\[
    \lim_{X\to\infty}\frac{1}{\log X}\int_{\substack{2\leq t\leq X \\ \pi(t;4,3)>\pi(t;4,1)}}\frac{dt}{t} = 0.9959\ldots
\]
Thus the ``bias'' toward primes of the form $4n+3$ is quite strong.  The literature on prime number races which study such inequities is quite vast.  See, for instance, the work of Fiorilli, Ford, Harper, Konyagin, Lamzouri, and Martin \cite{Fiorilli_Martin,FK, FLK, FHL}.

We use Theorem \ref{thm:densityfunction} to study inequities not between primes in different residue classes,
but between zeros of different $L$-functions.   Define
\[
    \mathscr{C}^{r}(\T^n) := \{h \in C^{r}(\T^n)\colon \textup{$h$ is nonnegative and not identically zero}\}.
\]
Let $h\in\mathscr{C}^{n+2}(\T^n)$, and let $\bm{\alpha}$ satisfy \eqref{eq:condition_alpha}.  We consider two holomorphic cusp forms $f_1$ and $f_2$ with trivial nebentypus, where
\[
    f_j(z) = \sum_{n=1}^{\infty}\lambda_{f_j}(n)n^{\frac{k_j-1}{2}}e^{2\pi i n z}\in S_{k_j}^{\mathrm{new}}(\Gamma_0(q_j)),\qquad j\in\{1,2\}
\]
have trivial nebentypus, even integral weights $k_j\geq 2$ and levels $q_j\geq 1$.  We assume that $f_j$ is normalized so that $\lambda_{f_j}(1)=1$ and that $f_j$ is an eigenfunction of all of the Hecke operators.  We call such cusp forms {\it newforms} (see \cite[Section 2.5]{Ono}).  It is classical that there exists $\pi_j\in\mathscr{A}$ such that $L(s,f_j)=L(s,\pi_j)$, so $\lambda_{f_j}(n) = \lambda_{\pi_j}(n)$, $g_{f_j,\bm{\alpha}}(\bm{t})=g_{\pi_j,\bm{\alpha}}(\bm{t})$, etc.

We say $f_1$ wins the $(\bm{\alpha}, h)$-race (against $f_2$) if for all large $T$, we have
\[
    \sum_{|\gamma_1| \le T} h(\bm{\alpha} \gamma_1)> \sum_{|\gamma_2| \le T} h(\bm{\alpha} \gamma_2),
\]
where $\gamma_j$ runs over the imaginary parts of the nontrivial zeros of $L(s, f_j)$.  If neither $f_1$ nor $f_2$ wins the $(\bm{\alpha}, h)$-race, we say the race is undecided.  The results of \cite{LZ} suggest that the winner of a decided $(\bm{\alpha},h)$-race is determined by the levels $q_1$ and $q_2$ and the behavior of $g_{f_1, \bm{\alpha}}$ and $g_{f_2, \bm{\alpha}}$.

First, we show that if $q_1 = q_2$ and $h \in \mathscr{C}^{3}(\T)$,
then the proportion of primes $p$ such that $f_1$ wins the $(\frac{\log p}{2\pi}, h)$-race is $\frac{1}{2}$.

\begin{corollary}\label{cor:equalconductors}
    For $j=1,2$, let $f_j  \in S_{k_j}^{\mathrm{new}} (\Gamma_0(q_j))$ be non-CM newforms.  Suppose that $f_1\neq f_2\otimes\chi$ for all primitive Dirichlet characters $\chi$.  If $h \in \mathscr{C}^{3}(\T)$, then
    \[
    \frac{\#\{ p \le X \colon \text{$f_1$ wins the $(\frac{\log p}{2\pi}, h)$-race}\}}{\#\{p\leq X\}} = \frac{1}{2}+O\Big(\frac{\sqrt{\log\log\log X}}{(\log\log X)^{1/4}}\Big).
    \]
\end{corollary}

Next we look at the distribution of
\[
H(f_1, f_2, h, \alpha) := \lim_{T \to \infty} \frac{\sum_{|\gamma_1| < T} h( \alpha \gamma_1 )
    - \sum_{|\gamma_2| < T} h(\alpha \gamma_2 ) }{T}
    - \frac{1}{\pi}\log\frac{q_{1}}{q_{2}}\int_{\mathbb{T}}h(t)dt
\]
as $\alpha$ varies over values of $\frac{\log p}{2\pi}$ for prime values of $p$.  Given $I \subseteq [-2, 2]$, let $\mu_{\mathrm{ST}}(I)$ be the Sato--Tate measure $1/(2\pi)\int_{I}\sqrt{4-t^2}dt$,
and let $\mu_{ST,2}$ be the product measure defined on boxes $I_1\times I_2\subseteq [-2,2]\times[-2,2]$ by $\mu_{\mathrm{ST},2}(I_1\times I_2)=\mu_{\mathrm{ST}}(I_1)\mu_{\mathrm{ST}}(I_2)$. For $\mathcal{I} \subseteq [-4,4]$, we define
\begin{equation}
\label{eq:nu_measure}
    \nu(\mathcal{I}) := \mu_{ST,2} ( \{ (x,y) \in [-2,2]^2: x - y \in \mathcal{I} \} ).
\end{equation}
For $h \in \mathscr{C}^3(\T)$, we set
    \[
        k_h := \int_0^1 h(t) \cos(2 \pi t) dt.
    \]

In the following statement and throughout the paper, for any interval $I$ and $\beta \in \mathbb{R}$, we let $\beta I = \{ \beta x: x \in I \}$.

\begin{corollary}\label{thm:STracedist}
    For $j=1,2$, let $f_j  \in S_{k_j}^{\mathrm{new}} (\Gamma_0(q_j))$ be normalized holomorphic non-CM newforms with trivial nebentypus.  Suppose that $f_1\neq f_2\otimes\chi$ for all primitive Dirichlet characters $\chi$.      Let $$\epsilon_X := \frac{(\log\log\log X)^{1/4}}{(\log\log X)^{1/8}}. $$ Assume $h \in \mathscr{C}^3(\T)$ is such that $k_h \neq 0$.
    Then for any interval $\mathcal{I}\subseteq[-4,4]$, we have
\[
\frac{\# \{ p \in [(1-\epsilon_X)X, X]\colon \frac{\log{X}}{\sqrt{X}} H(f_1, f_2, h, \frac{\log{p}}{2 \pi}) \in \mathcal{I} \}}{\#\{p\in[(1-\epsilon_X)X,X]\}}= \nu\Big(\frac{\pi}{2 k_h} \mathcal{I}\Big)+ O( \epsilon_X )
\]
    with an implied constant independent of $\mathcal{I}$.
\end{corollary}

If $q_1 > q_2$, then in contrast to Corollary \ref{cor:equalconductors},
$f_1$ wins the $(\frac{\log p}{2\pi},h)$-race for all except finitely many primes $p$.

\begin{corollary}
\label{thm:distinctconductors}
For $j=1,2$, let $f_j  \in S_{k_j}^{\mathrm{new}} (\Gamma_0(q_j))$ be normalized holomorphic newforms with trivial nebentypus.  If $q_1>q_2$, then for any $h \in \mathscr{C}^3(\T)$, there are at most finitely many primes $p$ such that $\pi_2$ wins the $(\frac{\log p}{2\pi},h)$-race.
\end{corollary}

Corollary \ref{thm:distinctconductors} shows that it is rare for $\pi_1$ to win an $(\alpha,h)$ race against $\pi_2$ if $q_2>q_1$, but we can show this occurs infinitely often.
Our result can be stated neatly in terms of \emph{local races} rather than in terms of the $(\alpha, h)$-races described above.
For $t_0 \in \mathbb{T}$, we say that $f_1$ wins the local
$(\alpha, t_0)$-race against $f_2$ if there exists a neighborhood $U$ of $t_0 \in \mathbb{T}$ such that
the $(\alpha, h)$-race is won by $f_1$ for all $h \in \mathscr{C}^{3}(\T)$ which are supported on $U$.

\begin{corollary}
    \label{thm:exceptional}
For $j=1,2$, let $f_j  \in S_{k_j}^{\mathrm{new}} (\Gamma_0(q_j))$ be normalized holomorphic non-CM newforms with trivial nebentypus.  Suppose that $f_1\neq f_2\otimes\chi$ for all primitive Dirichlet characters $\chi$.  Fix $t_0 \in [0,1)$.  Let $q_1$ be sufficiently large, and let $q_2 \in (q_1, q_1 + q_1^{1/2}]$, and let $k_1$ and $k_2$ be fixed.  There exists $\alpha \in \R$ such that $f_1$ wins the local $(\alpha, t_0)$-race against $f_2$.
\end{corollary}

In addition to Theorem \ref{thm:densityfunction}, the proofs of these corollaries rely on a quantifiable understanding of the joint distribution of $\lambda_{f_1}(p)$ and $\lambda_{f_2}(p)$ as $p$ varies over the primes.    Such an understanding follows from the effective version of the Sato--Tate conjecture which counts the number of primes $p\leq X$ such that $(\lambda_{f_1}(p),\lambda_{f_2}(p))\in I_1\times I_2$ proved by the third author in \cite{Thorner} (see Theorem \ref{thm:effectiveST} below).  One can prove analogues of Corollaries \ref{cor:equalconductors}, \ref{thm:STracedist}, and \ref{thm:exceptional} for Dirichlet $L$-functions by replacing the effective Sato-Tate estimates with results on primes in arithmetic progressions.



\section{Preliminaries}
\subsection{\texorpdfstring{$\operatorname{GL}_2$ $L$}{GL2 L}-functions over \texorpdfstring{$\mathbb{Q}$}{Q}}

Let $\pi\in\mathscr{A}$.  Here, we state the essential properties of $L$-functions of $L(s,\pi)$ that we use throughout our proofs.  See \cite[Chapter 5]{IK} for a convenient summary.  Given $\pi\in\mathscr{A}$ with level $q_{\pi}$, there exist suitable complex numbers $\alpha_{1,\pi}(p)$ and $\alpha_{2,\pi}(p)$ such that
\begin{equation*}
L(s,\pi) = \prod_{\textup{$p$ prime}}\prod_{j=1}^{2} ( 1-  \alpha_{j,\pi}(p) p^{-s} )^{-1} = \sum_{n=1}^{\infty}\frac{\lambda_{\pi}(n)}{n^s}.
\end{equation*}
The sum and product both converge absolutely for $\re(s)>1$.  There also exist spectral parameters $\kappa_{\pi}(1)$ and $\kappa_{\pi}(2)$ such that if we define
\[
L(s,\pi_{\infty})=\pi^{-s}\Gamma\Big(\frac{s+\kappa_{\pi}(1)}{2}\Big)\Gamma\Big(\frac{s+\kappa_{\pi}(2)}{2}\Big),
\]
then the completed $L$-function $\Lambda(s,\pi):=q_{\pi}^{s/2}L(s,\pi)L(s,\pi_{\infty})$ is entire of order 1.

Let $\tilde{\pi}\in\mathscr{A}$ be the contragredient representation.  We have $\alpha_{j,\tilde{\pi}}(p)=\overline{\alpha_{j,\pi}(p)}$ and $\kappa_{\tilde{\pi}}(j)=\overline{\kappa_{\pi}(j)}$ for $j=1,2$.    Moreover, there exists a complex number $W(\pi)$ of modulus 1 such that for all $s\in\mathbb{C}$, we have
\[
\Lambda(s,\pi)=W(\pi)\Lambda(1-s,\tilde{\pi}).
\]
Building on work of Kim and Sarnak \cite[Appendix]{KimSarnak}, Blomer and Brumley \cite{BB} proved that there exists $\theta\in[0,7/64]$ such that we have the uniform bounds
\begin{equation}
    \label{eqn:GRC}
    \log_p | \alpha_{j,\pi} (p)|,~-\re(\kappa_{\pi}(j))\leq \theta.
\end{equation}
The generalized Ramanujan conjecture and the Selberg eigenvalue conjecture assert that \eqref{eqn:GRC} holds with $\theta=0$.

The Rankin--Selberg $L$-function\footnote{The $\doteq$ suppresses the more complicated Euler factors at primes $p|q_{\pi}$. Their explicit description does not arise in our proofs.}
\[
L(s,\pi\otimes\tilde{\pi})=\sum_{n=1}^{\infty}\frac{\lambda_{\pi\times\tilde{\pi}}(n)}{n^s}\doteq \prod_{p\nmid q_{\pi}}\prod_{j=1}^{2}\prod_{j'=1}^{2}(1-\alpha_{j,\pi}(p)\overline{\alpha_{j',\pi}(p)}p^{-s})^{-1},\qquad \re(s)>1
\]
factors as $\zeta(s)L(s,\mathrm{Ad}^2\pi)$, where the adjoint square lift $\mathrm{Ad}^2\pi$ is an automorphic representation of $\GL_3(\mathbb{A}_{\Q})$.  Thus, $L(s,\mathrm{Ad}^2\pi)$ is an entire automorphic $L$-function.  This fact and the bound $|\lambda_{\pi}(n)|^2\leq \lambda_{\pi\times\tilde{\pi}}(n)$ \cite[Lemma 3.1]{JLW}, enable us to prove via contour integration that
\begin{equation}
\label{eqn:RS_ramanujan_average}
\sum_{n\leq X}|\lambda_{\pi}(n)|^2\leq \sum_{n\leq X}\lambda_{\pi\times\tilde{\pi}}(n)\ll_{\pi}X.
\end{equation}
It follows from \cite[Theorem 5.42]{IK} applied to $\zeta(s)$ and $L(s,\mathrm{Ad}^2\pi)$ that there exists an effectively computable constant $c_{\pi}>0$ such that $L(s,\pi\times\tilde{\pi})\neq 0$ in the region
\begin{equation}
    \label{eqn:ZFR_RS}
    \re(s)\geq 1- \frac{c_{\pi}}{\log(|\im(s)|+3)}.
\end{equation}

\subsection{Holomorphic newforms}
\label{sec:holomorphic_newforms}

Many of our corollaries pertain specifically to $\pi\in\mathscr{A}$ corresponding to holomorphic newforms.  As above, let
\[
    f(z) = \sum_{n=1}^{\infty}\lambda_f(n)n^{\frac{k-1}{2}}e^{2\pi inz}\in S_k^{\mathrm{new}}(\Gamma_0(q))
\]
be a holomorphic cuspidal newform (normalized so that $\lambda_f(1)=1$) of even integral weight $k\geq 2$, level $q\geq 1$, and trivial nebentypus.  If $\pi_f\in\mathscr{A}$ corresponds with $f$, then $L(s,f)=L(s,\pi_f)$, $\lambda_f(n)=\lambda_{\pi_f}(n)$, etc.

For these newforms, it follows from Deligne's proof of the Weil conjectures that the generalized Ramanujan conjecture holds.  Thus, for $f$ a holomorphic cuspidal newform as above, we may take $\theta=0$ in \eqref{eqn:GRC}.  Since we assume that $f$ has trivial central character, Deligne's bound implies that there exists $\theta_p\in[0,\pi]$ such that
\[
    \lambda_f(p)=2\cos\theta_p.
\]
The Sato--Tate conjecture, now a theorem due to Barnet-Lamb, Geraghty, Harris, and Taylor \cite{BLGHT}, states that the sequence $(\theta_p)$ is equidistributed in the interval $[-2,2]$ with respect to the measure $\frac{2}{\pi}(\sin t)^2 dt$.  In other words, if $I\subseteq[0,\pi]$ is a subinterval, then
\[
    \lim_{X\to\infty}\frac{\#\{p\leq X\colon \theta_p\in I\}}{\#\{p\leq X\}}=\frac{2}{\pi}\int_{I}(\sin t)^2 dt,\qquad \pi(X):=\#\{\textup{$p$ prime: }p\leq X\}.
\]
After a change of variables, this implies that for any interval $I \subseteq [-2,2]$, we have
\begin{equation}
    \label{eqn:Sato_Tate_1}
    \lim_{X \to \infty} \frac{ \# \{ \lambda_f(p) \in I: p \le X \}}{\#\{p\leq X\}} = \frac{1}{2\pi} \int_I \sqrt{4-t^2} dt =: \mu_{\mathrm{ST}}(I).
\end{equation}
A recent paper by Thorner \cite{Thorner} provides both an unconditional and a GRH-conditional rate of convergence in \eqref{eqn:Sato_Tate_1}.

Our corollaries of Theorem \ref{thm:densityfunction} require a natural refinement of the Sato--Tate conjecture.  For $j=1,2$, let $f_j\in S_{k_j}^{\mathrm{new}}(\Gamma_0(q_j))$ be a holomorphic cuspidal newform as above.  Suppose that $f_1\neq f_2\otimes\chi$ for all primitive nontrivial Dirichlet characters $\chi$.  Building on work of Harris \cite{Harris_2}, Wong \cite{Wong_2} proved that the sequences $(\lambda_{f_1}(p))$ and $(\lambda_{f_2}(p))$ exhibit a joint distribution:  If $I_1,I_2\subseteq [-2,2]$, then
\begin{equation}
\label{eqn:ST_2}
\lim_{X\to\infty}\frac{\#\{p\leq X\colon \lambda_{f_1}(p)\in I_1,~\lambda_{f_2}(p)\in I_2\}}{\#\{p\leq X\}}=\mu_{\mathrm{ST}}(I_1)\mu_{\mathrm{ST}}(I_2).
\end{equation}
Our corollaries of Theorem \ref{thm:densityfunction} require a nontrivial unconditional bound on the rate of convergence in \eqref{eqn:ST_2}.  Such a bound was recently proved in \cite[Theorem 1.2]{Thorner}.

\begin{theorem}
\label{thm:effectiveST}
    For $j=1,2$, let $f_j \in S_{k_j}^{\mathrm{new}}(\Gamma_0(q_j))$ be a normalized holomorphic cuspidal newform with trivial nebentypus.   Suppose that $f_1\neq f_2\otimes\chi$ for all primitive nontrivial Dirichlet character $\chi$.  Let $I_1, I_2\subseteq [-2, 2]$ be subintervals.  There exists an absolute and effectively computable constant $c>0$ such that
    \[
     \Big|\frac{\# \{ p \le X\colon \lambda_{f_1, p} \in I_1, \lambda_{f_2}(p) \in I_2 \}}{\#\{p\leq X\}} - \mu_{\mathrm{ST}}(I_1) \mu_{\mathrm{ST}}(I_2)\Big|\leq c\frac{\log \log(k_1 k_2 q_1 q_2 \log X)}{\sqrt{\log \log X}}.
    \]
\end{theorem}

We use the following result to prove our corollaries of Theorem \ref{thm:densityfunction}.

\begin{corollary}\label{cor:STlimit}
    Let  $f_1$ and $f_2$ be as in Theorem \ref{thm:effectiveST}.  Recall the definition of $\nu$ in \ref{eq:nu_measure}.  If $\mathcal{I}  \subseteq[-4,4]$, then
    \[
    \frac{\# \{p \le X: \lambda_{f_1}(p) - \lambda_{f_2}(p) \in \mathcal{I} \}}{\#\{p\leq X\}}=\nu(\mathcal{I})+O\Big(\frac{\sqrt{\log\log\log X}}{(\log\log X)^{1/4}}\Big).
    \]
    with an implied constant independent of $\mathcal{I}$.
\end{corollary}

\begin{proof}
    Let
    \[
    g(X) := \frac{(\log\log X)^{1/4}}{\sqrt{\log\log\log X}},\qquad S := \{(x,y) \in [-2,2]^2: x-y \in \mathcal{I} \}.
    \]
    We will define rectangles whose unions approximate $S$.
    Let $x_j = -2 + j \lfloor \frac{4}{g(X)} \rfloor$.
    If $\mathcal{I} = ( c, d)$, set $R_j = [x_j, x_{j+1}] \times [x_j - d, x_{j+1} - c]$,
    and similarly $T_j := [x_j, x_{j+1}] \times [x_{j+1} - d, x_j - c]$.  By construction, we have $\cup_j T_j \subseteq S \subseteq \cup_j R_j$, which implies that
    \begin{align}\label{eq:squeeze}
    \# \{ p \le X: (\lambda_{f_1} (p), \lambda_{f_2}(p)) \in \cup T_j \} &\le \# \{ p \le X: (\lambda_{f_1} (p), \lambda_{f_2}(p)) \in S \} \notag\\
    & \le  \# \{ p \le X: (\lambda_{f_1} (p), \lambda_{f_2}(p)) \in \cup_j R_j \}
\end{align}

    We apply Theorem \ref{thm:effectiveST} to count the primes $p$ with $(\lambda_{f_1} (p), \lambda_{f_2}(p))$
    in $\cup_j R_j$.
    \begin{align*}
        \frac{\# \{ p \le X: (\lambda_{f_1} (p), \lambda_{f_2}(p)) \in \cup_j R_j \}}{\#\{p\leq X\}} & = \sum_{j = 1}^{g(X)} \Big( \mu_{ST, 2}(R_j)
        + O \Big( \frac{\log \log \log X}{\sqrt{\log \log X}} \Big) \Big) \\
        & = \mu_{ST, 2} (\cup_j R_j) + O\Big(g(X)  \frac{\log \log \log X}{\sqrt{\log \log X}}\Big).
    \end{align*}

      Since the area of $\cup_j R_j \setminus S$ is at most $g(X) \lfloor \frac{4}{g(X)} \rfloor^2$, we have
    \[
        \mu_{ST, 2}(\cup_j R_j) - \mu_{ST, 2}(S) = \mu_{ST, 2}( (\cup_j R_j ) \setminus S) = O(g(X)^{-1}),
    \]
    and we conclude that
    \begin{equation}
    \label{eqn:ST_applied}
    \begin{aligned}
        &\frac{\# \{ p \le X: (\lambda_{f_1} (p), \lambda_{f_2}(p)) \in \cup_j R_j \}}{\#\{p\leq X\}} \\
        &= \nu(\mathcal{I}) + O \Big( g(X)^{-1} +  g(X) \frac{\log \log \log X}{\sqrt{\log \log X}} \Big)= \nu(\mathcal{I})  + O\Big(\frac{\sqrt{\log\log\log X}}{(\log\log X)^{1/4}}\Big).
    \end{aligned}
    \end{equation}
    The same argument shows that \eqref{eqn:ST_applied} holds with $R_j$ replaced with $T_j$ on the left hand side.  The lemma now follows from (\ref{eq:squeeze}).
    \end{proof}

Finally, we require a refinement of \eqref{eq:number_of_zeros}, namely
\begin{equation}\label{eq:zerocount}
    N_f(T) = T\log\Big(q_f \Big(\frac{T}{2\pi e}\Big)^2\Big)+O(\log(k_f q_f T)).
\end{equation}
This is \cite[Theorem 5.8]{IK} applied to $L(s,f)$.

\section{Proof of Theorem \ref{thm:ZDE}}

Let $\pi\in\mathscr{A}$.  We detect the zeros of $L(s,\pi)$ by estimating a mollified second moment of $L(s,\pi)$ near the line $\re(s)=\frac{1}{2}$.  To describe our mollifier, we define $\mu_{\pi}(n)$ by the convolution identity
\begin{equation}
\label{eqn:mob_def}
\sum_{d|n}\mu_{\pi}(d)\lambda_{\pi}(n/d)=\begin{cases}
    1&\mbox{if $n=1$,}\\
    0&\mbox{if $n>1$}
\end{cases}
\end{equation}
so that
\[
\frac{1}{L(s,\pi)}=\sum_{n=1}^{\infty}\frac{\mu_{\pi}(n)}{n^s},\qquad \mathrm{Re}(s)>1.
\]
Let $T>0$ be a large parameter, and let $0<\varpi<\frac{1}{4}$.  Define $P(t)$ by
\begin{equation}
\label{eqn:P(t)_def}
    P(t) = \begin{cases}
        1&\mbox{if $0\leq t\leq T^{\varpi/2}$,}\\
      2(1-\frac{\log t}{\log T^{\varpi}})&\mbox{if $T^{\varpi/2}<t\leq T^{\varpi}$,}\\
        0&\mbox{if $t>T^{\varpi}$.}
        \end{cases}
\end{equation}
Our mollifier is
\begin{equation}
\label{eqn:mollifier_def}
M_{\pi}(s, T^{\varpi}) = \sum_{n\leq T^{\varpi}}\frac{\mu_{\pi}(n)}{n^{s}}P(n).
\end{equation}
As a proxy for detecting zeros of $L(s,\pi)$ near $\mathrm{Re}(s)=\frac{1}{2}$, where $L(s,\pi)$ oscillates wildly, we detect zeros of the mollified $L$-function $L(s,\pi)M_{\pi}(s, T^{\varpi})$ near $\mathrm{Re}(s)=\frac{1}{2}$.  To this end, we let $w:\R\to\R$ be an infinitely differentiable function whose support is a compact subset of $[T/4,2T]$ and whose $j$-th derivative satisfies $|w^{(j)}(t)|\ll_{w,j}((\log T)/T)^j$ for all $j\geq 0$.  Also, let $w(t)=1$ for $t\in[T/2,T]$.  We will estimate
\[
I_f(\alpha,\beta)=\int_{-\infty}^{\infty}w(t)L(\tfrac{1}{2}+\alpha+it,\pi)L(\tfrac{1}{2}+\beta-it,\tilde{\pi} )|M_{\pi}(\tfrac{1}{2}+\tfrac{1}{\log T}+it,T^{\varpi})|^2 dt,
\]
eventually choosing $\alpha$ and $\beta$ to equal $1/\log T$.  Define
\begin{align*}
G(s) &= e^{s^2}\frac{(\alpha+\beta)^2-(2s)^2}{(\alpha+\beta)^2},    \\
g_{\alpha,\beta}(s,t)&=\frac{L(\frac{1}{2}+\alpha+s+it,\pi_{\infty})L(\frac{1}{2}+\beta+s-it,\tilde{\pi}_{\infty})}{L(\frac{1}{2}+\alpha+it,\pi_{\infty})L(\frac{1}{2}+\beta-it,\tilde{\pi}_{\infty})}\\
V_{\alpha,\beta}(s,t)&=\frac{1}{2\pi i}\int_{\re(s)=1}\frac{G(s)}{s}g_{\alpha,\beta}(s,t)x^{-s}ds,\\
X_{\alpha,\beta}(t)&=\frac{L(\frac{1}{2}-\alpha-it,\pi_{\infty})L(\frac{1}{2}-\beta+s+it,\tilde{\pi}_{\infty})}{L(\frac{1}{2}+\alpha+it,\pi_{\infty})L(\frac{1}{2}+\beta-it,\tilde{\pi}_{\infty})}.
\end{align*}
As in \cite{Bernard,AT}, it follows from the approximate functional equation (see \cite[Section 5.2]{IK} also) that
\[
I_f(\alpha,\beta)=\sum_{a,b\leq T^{\varpi}}\frac{\mu_{\pi}(a)\overline{\mu_{\pi}(b)}}{\sqrt{ab}}\frac{P(a)P(b)}{(ab)^{\frac{1}{\log T}}}(D_{a,b}^{+}(\alpha,\beta)+D_{a,b}^{-}(\alpha,\beta)+N_{a,b}^{+}(\alpha,\beta)+N_{a,b}^{-}(\alpha,\beta)),
\]
where we have split the sum into diagonal terms
\begin{align*}
    D_{a,b}^{+}(\alpha,\beta) &= \sum_{am=bn} \frac{\lambda_{\pi}(m)\overline{\lambda_{\pi}(n)}}{m^{\frac{1}{2}+ \alpha}n^{\frac{1}{2}+\beta}} \int_{-\infty}^{\infty} w(t)  V_{\alpha,\beta}(mn,t) \, dt, \\
    D_{a,b}^{-}(\alpha,\beta) &= \sum_{am=bn} \frac{\lambda_{\pi}(m)\overline{\lambda_{\pi}(n)}}{m^{\frac{1}{2}- \beta}n^{\frac{1}{2}-\alpha}} \int_{-\infty}^{\infty} w(t)  X_{\alpha,\beta}(t) V_{-\beta,-\alpha}(mn,t) \, dt
\end{align*}
and off-diagonal terms
\begin{align*}
    N_{a,b}^{+}(\alpha,\beta) &= \sum_{am\neq bn} \frac{\lambda_{\pi}(m)\overline{\lambda_{\pi}(n)}}{m^{\frac{1}{2}+ \alpha}n^{\frac{1}{2}+\beta}} \int_{-\infty}^{\infty} w(t) \pfrac{bn}{am}^{it} V_{\alpha,\beta}(mn,t) \, dt, \\
    N_{a,b}^{-}(\alpha,\beta) &= \sum_{am\neq bn} \frac{\lambda_{\pi}(m)\overline{\lambda_{\pi}(n)}}{m^{\frac{1}{2}- \beta}n^{\frac{1}{2}-\alpha}} \int_{-\infty}^{\infty} w(t) \pfrac{bn}{am}^{it} X_{\alpha,\beta}(t) V_{-\beta,-\alpha}(mn,t) \, dt.
\end{align*}

\begin{lemma}
\label{lem:off_diag}
    Let $\epsilon>0$.  If $\alpha,\beta\in\mathbb{C}$ satisfy $|\alpha|,|\beta|\ll \frac{1}{\log T}$ and $|\alpha+\beta|\gg\frac{1}{\log T}$, then for any integers $a,b\geq 1$, we have that $N_{a,b}^{\pm}(\alpha,\beta)\ll_{\varepsilon}(ab)^{\frac{1}{2}}T^{\frac{1}{2}+\theta}(abT)^{\varepsilon}$.
\end{lemma}
\begin{proof}
    This is \cite[Proposition 3.4]{AT}.
\end{proof}

\begin{corollary}
    \label{cor:off_diagonal}
    Fix $0< \varpi< \frac{1}{4}-\frac{\theta}{2}$.  There exists a constant $\delta>0$ such that
    \[
    \Big|\sum_{a,b\leq T^{\varpi}}\frac{\mu_{\pi}(a)\overline{\mu_{\pi}(b)}}{\sqrt{ab}}\frac{P(a)P(b)}{(ab)^{\frac{1}{\log T}}}(N_{a,b}^{+}(\alpha,\beta)+N_{a,b}^{-}(\alpha,\beta))\Big|\ll T^{1-\delta}.
    \]
\end{corollary}
\begin{proof}
    Let $\epsilon>0$.  First, observe that $|\mu_{\pi}(a)|\ll 1+|\lambda_{\pi}(a)|\ll 1+|\lambda_{\pi}(a)|^2$.
    We then apply Lemma \ref{lem:off_diag} and bound everything else trivially to obtain
    \[
        \sum_{a,b\leq T^\varpi} \frac{\mu_{\pi}(a)\overline{\mu_{\pi}(b)}}{\sqrt{ab}} \frac{P(a) P(b)}{(ab)^{\frac{1}{\log T}}} \sum_{\pm} N_{a,b}^{\pm}(\alpha,\beta) \ll T^{\frac{1}{2}+\theta+\epsilon} \Big(\sum_{a\leq T^\varpi} 1 + \sum_{a\leq T^\varpi} |\lambda_{\pi}(a)|^2 \Big)^2.
    \]
    By \eqref{eqn:RS_ramanujan_average}, this is $\ll_{f,\epsilon} T^{\frac{1}{2}+\theta + 2\varpi +\epsilon}$.   If $\varpi \leq \frac14 - \frac\theta 2-\epsilon$, then the above display is $\ll T^{1-\epsilon}$.
\end{proof}

\begin{proposition}
\label{prop:AndersenThorner}
If $T$ is sufficiently large and $\varpi\in(0, \frac{1}{4}-\frac{\theta}{2})$ is fixed, then
    \[
    \int_{T/2}^{T}|L(\tfrac{1}{2}+\tfrac{1}{\log T}+it,\pi)M_{\pi}(\tfrac{1}{2}+\tfrac{1}{\log T}+it,T^{\varpi})|^2 dt\ll T.
    \]
\end{proposition}
\begin{proof}
In light of Corollary \ref{cor:off_diagonal}, it remains to bound the diagonal contribution.  We first note by a calculation identitcal to \cite[Lemma 11]{Bernard} that
\begin{multline*}
\sum_{a,b\leq T^{\varpi}}\frac{\mu_{\pi}(a)\overline{\mu_{\pi}(b)}}{\sqrt{ab}}\frac{P(a)P(b)}{(ab)^{\frac{1}{\log T}}}(D_{a,b}^{+}(\alpha,\beta)+D_{a,b}^{-}(\alpha,\beta))\\
=\sum_{a,b\leq T^{\varpi}}\frac{\mu_{\pi}(a)\overline{\mu_{\pi}(b)}}{\sqrt{ab}}\frac{P(a)P(b)}{(ab)^{\frac{1}{\log T}}}(D_{a,b}^{+}(\alpha,\beta)+T^{-2(\alpha+\beta)}D_{a,b}^{+}(-\alpha,-\beta))+O\Big(\frac{T}{\log T}\Big).
\end{multline*}
So it suffices for us to estimate
\[
I^D(\alpha,\beta)=\sum_{a,b\leq T^{\varpi}}\frac{\mu_{\pi}(a)\overline{\mu_{\pi}(b)}}{\sqrt{ab}}\frac{P(a)P(b)}{(ab)^{\frac{1}{\log T}}}D_{a,b}^{+}(\alpha,\beta)
\]
with $\alpha=\beta=1/\log T$.

Let $\sigma_0 = \frac{1}{2}+\frac{1}{\log T}$.  We observe via the Mellin inversion that $I^D(\alpha,\beta)$ equals
\begin{align*}
\frac{4}{(\log (T^{\varpi}))^2}\int_{-\infty}^{\infty}\frac{1}{(2\pi i)^3}\int_{(1)}\int_{(1)}\int_{(1)}T^{\frac{\varpi(u+v)}{2}}(T^{\frac{u\varpi}{2}}-1)(T^{\frac{v\varpi}{2}}-1)\frac{G(s)}{s}g_{\alpha,\beta}(s,t)\\
\times\sum_{\substack{a,b,m,n\geq 1 \\ am=bn}}\frac{\mu_{\pi}(a)\overline{\mu_{\pi}(b)}\lambda_{\pi}(m)\overline{\lambda_{\pi}(n)}}{a^{\sigma_0+v}b^{\sigma_0+u}m^{\frac{1}{2}+\alpha+s}n^{\frac{1}{2}+\beta+s}}ds\frac{du}{u^2}\frac{dv}{v^2}dt.
\end{align*}
By a computation identical to \cite[Lemma 6]{Bernard}, there exists a product of half-planes containing an open neighborhood of the point $u=v=s=0$ and an Euler product $A_{\alpha,\beta}(u,v,s)$, absolutely convergent for $(u,v,s)$ in said product of half-planes, such that
\begin{multline*}
    \sum_{\substack{a,b,m,n\geq 1 \\ am=bn}}\frac{\mu_{\pi}(a)\overline{\mu_{\pi}(b)}\lambda_{\pi}(m)\overline{\lambda_{\pi}(n)}}{a^{\sigma_0+v}b^{\sigma_0+u}m^{\frac{1}{2}+\alpha+s}n^{\frac{1}{2}+\beta+s}}=A_{\alpha,\beta}(u+\tfrac{1}{\log T},v+\tfrac{1}{\log T},s)\\
\times \frac{L(1+\alpha+\beta+2s,\pi\otimes\tilde{\pi})L(1+\frac{2}{\log T}+u+v,\pi\otimes\tilde{\pi})}{L(1+\frac{1}{\log T}+\alpha+u+s,\pi\otimes\tilde{\pi})L(1+\frac{1}{\log T}+\beta+v+s,\pi\otimes\tilde{\pi})}.
\end{multline*}
By M{\"o}bius inversion and a continuity argument, one can prove that $A_{0,0}(0,0,0)=1$ \cite[Lemma 7]{Bernard}.

Upon choosing $\delta$ sufficiently small and shifting the contours to $\re(u)=\delta$, $\re(v)=\delta$, and $\re(s)=-\delta/2$, we find that
\begin{multline*}
  I^D(\alpha,\beta) = \frac{4L(1+\alpha+\beta,\pi\otimes\tilde{\pi})}{(2\pi i)^2(\log (T^{\varpi}))^2}\int_{\R}w(t)dt \int_{(\delta)}\int_{(\delta)}T^{\frac{\varpi(u+v)}{2}}(T^{\frac{u\varpi}{2}}-1)(T^{\frac{v\varpi}{2}}-1)\\
    \times \frac{L(1+\frac{2}{\log T}+u+v,\pi\otimes\tilde{\pi})A_{\alpha,\beta}(u+\tfrac{1}{\log T},v+\tfrac{1}{\log T},0)}{L(1+\frac{1}{\log T}+\alpha+u,\pi\otimes\tilde{\pi})L(1+\frac{1}{\log T}+\beta+v,\pi\otimes\tilde{\pi})}\frac{du}{u^2}\frac{dv}{v^2}+O\Big(\frac{T}{\log T}\Big).
\end{multline*}
The contribution from the pole at $u=v=0$ determines the magnitude of the double integral over $u$ and $v$, and this magnitude is $\asymp_{\pi}(\log T)^2$ for $\alpha$ and $\beta$ in our prescribed range. (This follows once we push the $u$- and $v$-contours to the left using \eqref{eqn:ZFR_RS}.)  Since $\int_{\R}w(t)dt\asymp T$ by hypothesis, we combine all preceding contributions and choose $\alpha=\beta=1/\log T$ to conclude the desired bound.
\end{proof}

\begin{corollary}
    \label{cor:AndersenThorner}
If $T$ is sufficiently large and $\varpi\in(0, \frac{1}{4}-\frac{\theta}{2})$ is fixed, then
    \[
    \int_{T/2}^{T}|L(\tfrac{1}{2}+\tfrac{1}{\log T}+it,\pi)M_{\pi}(\tfrac{1}{2}+\tfrac{1}{\log T}+it,T^{\varpi})-1|^2 dt\ll T.
    \]
\end{corollary}

\begin{proof}
This follows from Proposition \ref{prop:AndersenThorner} and the Cauchy--Schwarz inequality.
\end{proof}

We prove a corresponding estimate on a vertical line to the right of $\mathrm{Re}(s)=1$.

\begin{lemma}
    \label{lem:MVT_Dirichlet}
If $\varpi\in(0, \frac{1}{4}-\frac{\theta}{2})$ and $A>\theta$ are fixed and $T$ is sufficiently large, then
\[
\int_{T/2}^T|L(1+A+it,\pi)M_{\pi}(1+A+it,T^{\varpi})-1|^2 dt\ll_{A} T^{1+(\theta-A-\frac{1}{2})\varpi}(\log T)^{15}.
\]
\end{lemma}
\begin{proof}
Let $A>\theta$.  It follows immediately from \cite[Corollary 3]{MV_Hilbert} that if $T>1$ and $(b_n)$ is any sequence of complex numbers satisfying $\sum_{n}n|b_n|^2<\infty$, then
    \begin{equation}
    \label{eqn:MV_Hilbert}
    \int_{T/2}^{T}\Big|\sum_{n=1}^{\infty}b_n n^{-it}\Big|^2 dt=\sum_{n=1}^{\infty}|b_n|^2(T/2+O(n)).
    \end{equation}
We define $a_n$ by the identity
\begin{align*}
L(1+A+it,\pi)M_{\pi}(1+A+it,T^{\varpi})-1 = \sum_{n=1}^{\infty}a_n n^{-1-A-it}.
\end{align*}
By \eqref{eqn:mob_def}, \eqref{eqn:P(t)_def}, and \eqref{eqn:mollifier_def}, we have that $a_n=0$ for all $n\leq T^{\varpi/2}$ and $|a_n|^2\ll (n^{\theta}d_4(n))^2\ll n^{2\theta}d_{16}(n)$, where $d_k(n)$ is the $n$-th Dirichlet coefficient of $\zeta(s)^k$.  The desired result follows once we apply \eqref{eqn:MV_Hilbert} with $b_n = a_n n^{-1-A}$.
\end{proof}

\begin{proof}[Proof of Theorem \ref{thm:ZDE}]
    We use Gabriel's convexity principle \cite[\S 7.8]{titchmarsh2} to interpolate the bounds in Corollary \ref{cor:AndersenThorner} and Lemma \ref{lem:MVT_Dirichlet}.  In particular, if $0<\varpi<\frac{1}{4}-\frac{\theta}{2}$ and $A>\theta$ are fixed and $c=\varpi(1+2A-2\theta)/(1+2A)$, then in the range $\frac{1}{2}+\frac{1}{\log T}\leq\sigma\leq 1+A$, we have
    \begin{equation}
    \label{eq:interpolation}
        \begin{aligned}
            &\int_{T/2}^T|L(\sigma+it,\pi)M_{\pi}(\sigma +it,T^{\varpi})-1|^2 dt\\
&\ll T^{\frac{1+A-\sigma}{1+A-(\frac{1}{2}+\frac{1}{\log T})}}(T^{1+(\theta-A-\frac{1}{2})\varpi}(\log T)^{15})^{1-\frac{1+A-\sigma}{1+A-(\frac{1}{2}+\frac{1}{\log T})}}\ll T^{ 1 - c(\sigma - \frac{1}{2})}.
        \end{aligned}
    \end{equation}
    Define $\Phi(s) := 1- (1- L(s,\pi) M_{\pi}(s,T^{\varpi}))^2$.  By construction, if $\alpha\in\mathbb{C}$, then
    \[
    \mathop{\mathrm{ord}}_{s=\alpha}\Phi(s)\geq \mathop{\mathrm{ord}}_{s=\alpha}L(s,\pi)M_{\pi}(s,T^{\varpi}).
    \]
    For any $M \ge 1$, let $C_M$ be the rectangular contour with corners $\sigma + iT/2$,  $\sigma + iT$, $M + iT/2$, and $M + i T$.  Applying Littlewood's Lemma (\cite{Titchmarsh}, pp. 132--133) and letting $M \to \infty$, we have
    \begin{align*}
        & \int_{\sigma}^{1}(N_{\pi}(\sigma', T) - N_{\pi}(\sigma',\tfrac{T}{2}))d \sigma'\le \lim_{M \to \infty} \frac{1}{2 \pi} \int_{C_M} \log \Phi(s) ds \\
        & = \frac{1}{2 \pi} \int_{T/2}^{T} \log | \Phi(\sigma+it)| dt + \frac{1}{2\pi} \int_{\sigma}^{\infty} \arg( \Phi(x+i T)) dx
        - \frac{1}{2\pi} \int_{\sigma}^{\infty} \arg(\Phi(x+\tfrac{iT}{2}) )dx.
    \end{align*}
    In view of the bound $\log|1 + z| \le  |z|$, it follows from \eqref{eq:interpolation} that
    \[
        \frac{1}{2 \pi} \int_{T/2}^{T} \log | \Phi(\sigma+it)| dt\leq \int_{T/2}^T|L(\sigma+it,\pi)M_{\pi}(\sigma +it,T^{\varpi})-1|^2 dt\ll T^{ 1 - c(\sigma - \frac{1}{2})}.
    \]

    For the second and third integrals, we consider the integrals over $x \le 1$ and over $x >1$ separately.
    For $x \in (1, \infty)$ we can trivially bound the Dirichlet series as
    \[
        |1- L(x + iT,\pi) M_{\pi}(x+iT,T^{\varpi})| \le \sum_{n=2}^{\infty} |a_n| n^{-x}\ll 2^{-x},
    \]
    where $(a_n)$ is a certain sequence of complex numbers such that $|a_n|\ll_{\epsilon}n^{\theta+\epsilon}$ for any fixed $\epsilon>0$ and all $n\geq 2$.  Thus, we have
    \[
        |\arg( 1- (1- L(x + iT,\pi) M_{\pi}(x+iT,T^{\varpi}))^2))| \ll 2^{-x},
    \]
    so
    \[
        \Big|\int_1^{\infty} \arg(\Phi(x+iT) )dx\Big|,~\Big|\int_1^{\infty} \arg(\Phi(x+2iT) )dx\Big|\ll 1.
    \]
 To handle the integrals for $\frac{1}{2}+\frac{1}{\log T}\leq x\leq 1$, we use the trivial bound
    \[
        \Big|\int_{\sigma}^{1}\arg(\Phi(x+2iT))dx\Big|\leq (1-\sigma)\max_{\sigma\leq x\leq 1}|\arg(\Phi(x+2iT))|\ll \log T.
    \]
    A proof of the final bound is contained within the proof of \cite[Theorem 5.8]{IK}.  The corresponding integral for $\Phi(x+iT)$ has the same bound.

    By the preceding work, it follows by dyadic decomposition that
    \[
        \int_{\sigma}^{1} N_{\pi}(\sigma', T) d\sigma' \ll T^{1 - c(\sigma - \frac{1}{2})},\qquad \frac{1}{2}+\frac{1}{\log T}\leq\sigma\leq 1.
    \]
    This estimate, the mean value theorem for integrals, and the fact that $N_{\pi}(\sigma,T)$ is monotonically decreasing as $\sigma$ increases together imply that
    \begin{align*}
        N_{\pi}(\sigma,T)&\leq \frac{1}{\sigma-(\sigma-\frac{1}{\log T})}\int_{\sigma-\frac{1}{\log T}}^{\sigma}N_{\pi}(\sigma',T)d\sigma'\\
        &\ll \frac{1}{\sigma-(\sigma-\frac{1}{\log T})}\int_{\sigma-\frac{1}{\log T}}^{1}N_{\pi}(\sigma',T)d\sigma'\ll T^{1-c(\sigma-\frac{1}{2})}\log T.
    \end{align*}
    If $\frac{1}{2}\leq \sigma\leq \frac{1}{2}+\frac{1}{\log T}$, then \eqref{eq:number_of_zeros} implies that $N_{\pi}(\sigma,T)\ll T\log T\asymp T^{1-c(\sigma-\frac{1}{2})}\log T$.

    To finish the proof, note that if $\theta=0$, for all $A>0$, we have $c=\varpi$.  If $\theta>0$, then fix $0<\epsilon<\theta\varpi/(\theta+\frac{1}{2})$ and choose $A = \frac{\theta\varpi}{\epsilon}-\frac{1}{2}$.  With these choices, we find that $c>\varpi-\epsilon$.  Theorem \ref{thm:ZDE} now follows.
\end{proof}

\section{Proof of Theorem \ref{thm:densityfunction}}

We begin with a few preliminary lemmas. Throughout the section, $\theta$ is an admissible exponent
toward the generalized Ramanujan conjecture as in Theorem~\ref{thm:ZDE}.
Our first result is an $n$-dimensional version of the Riemann--Lebesgue lemma.

\begin{lemma}
\label{thm:coefficientbound}
    Let $J \geq 1$.  Suppose that $h\in C^{n+2}(\T^n)$ has the Fourier expansion
    \[
        h(\bm{t}) = \sum_{\bm{m} \in \mathbb{Z}^n} c_{\bm{m}} e^{2 \pi i (\bm{m} \cdot \bm{t})}.
    \]
  We have $|c_{\bm{m}}|\ll_h \|\bm{m}\|_2^{-n-2}$, and consequently, we have
    \[
        h(\bm{t}) = \sum_{\substack{\bm{m} \in \mathbb{Z}^n \\ \|\bm{m}\|_2\leq J}}
        c_{\bm{m}} e^{2 \pi i (\bm{m} \cdot \bm{t})} + O_h(J^{-2}).
    \]
\end{lemma}

\begin{proof}
    We have
    \begin{equation}\label{eq:cm}
        c_{\bm{m}} = \int_{\mathbb{T}^n} h(\bm{t}) e^{- 2 \pi i (\bm{m} \cdot \bm{t})} \textup{d}\bm{t}.
    \end{equation}
    Let $\bm{m} = (m_1, m_2, \ldots, m_n)$. Choose $j \in \{1, \ldots, n\}$ such that $|m_j|=\|\bm{m}\|_{\infty}$.
    Integrate \eqref{eq:cm} by parts for $n+2$ times with respect to the coordinate $t_j$ of $\bm{t}$ so that
    \begin{equation}
        \label{cm_bound}
        c_{\bm{m}} = \frac{(-1)^{n+2}}{(2\pi m_j)^{n+2}}\int_{\mathbb{T}^n} \Big( \frac{\partial^{n+2}}{\partial t_j^{n+2}} h(\bm{t}) \Big)
        e^{-2\pi i (\bm{m} \cdot \bm{t})} \textup{d}\bm{t}.
    \end{equation}
Since $\sqrt{n}\|\bm{m}\|_{\infty}\geq \|\bm{m}\|_2$, the desired result follows from the triangle inequality.
\end{proof}

\begin{lemma} \label{exp_sum_bound}
    Let $x > 1$ and $T \ge 2$, and let $\langle x\rangle$ being the closest integer to $x$.  We have
    \begin{align*}
        \sum_{|\gamma|\leq T} x^{\rho}
        = -\frac{\Lambda_{\pi}(\langle x \rangle)}{\pi} \cdot \frac{e^{iT \log(x / \langle x \rangle)} - 1}{i \log(x / \langle x \rangle)}
        + O \Big( x^{1 + \theta}(\log(2 x) + \log T) + \frac{\log T}{\log x} \Big).
    \end{align*}
\end{lemma}

\begin{proof}
    This is \cite[Lemma 2]{FSZ} with $\epsilon=\theta$.
\end{proof}

Using Theorem \ref{thm:ZDE} and Lemma \ref{exp_sum_bound}, we prove an analogue of \cite[(3.8)]{FZ}.

\begin{lemma} \label{substract_exp}
Let $c$ be as in Theorem \ref{thm:ZDE}.  If $1< x< \exp(\frac{c}{3}\frac{\log T}{\log\log T})$, then
    \[
        \sum_{|\gamma| \leq T} x^{i \gamma}  =  \sum_{|\gamma| \leq T}x^{\rho - \frac{1}{2}} +O\Big(\frac{T (\log x)^2}{\log T}+\frac{T}{(\log T)^2}\Big).
    \]
\end{lemma}

\begin{proof}
    Let $\delta = \frac{3}{c}\frac{\log\log T}{\log T}$, so $0 < \delta \log x < 1$. By Theorem \ref{thm:ZDE}, we have that
    \begin{multline*}
        \Big|\sum_{\substack{|\gamma|\leq T \\ |\beta - \frac{1}{2}| \ge \delta}} (x^{i \gamma} - x^{\rho - \frac{1}{2}})\Big|
        \ll \sum_{\substack{|\gamma|\leq T \\ \beta \ge \frac{1}{2} + \delta}} x^{\beta - \frac{1}{2}}
        \ll x^\delta N_{f}(\tfrac{1}{2} + \delta, T) + \log x \int_{\frac{1}{2} + \delta}^{1}
        x^{\sigma - \frac{1}{2}} N_{f}(\sigma, T) d \sigma \\
        \ll \frac{T}{(\log T)^2}.
    \end{multline*}
    By the functional equation for $L(s,\pi)$, $\beta+i\gamma$ is a nontrivial zero
    if and only if $1-\beta+i\gamma$ is a nontrivial zero.  Therefore, we have
    \begin{equation}
    \label{eqn:close_to_1/2}
    \begin{aligned}
        \Big|\sum_{\substack{|\gamma|\leq T \\ |\beta - \frac{1}{2}| < \delta}}
        (x^{i \gamma} - x^{\rho - \frac{1}{2}})\Big| &= \Big|\sum_{\substack{|\gamma|\leq T \\ \frac{1}{2}<\beta<\frac{1}{2}+\delta}}
        (x^{i \gamma}(1 - x^{\beta - \frac{1}{2}})+x^{i\gamma}(1-x^{\frac{1}{2}-\beta})\Big| \\
        & \leq \sum_{\substack{|\gamma|\leq T \\ 0<\beta-\frac{1}{2}<\delta}}
        |x^{i \gamma}(1 - x^{\beta - \frac{1}{2}})+x^{i\gamma}(1-x^{-(\beta-\frac{1}{2})})| \\
        & = \sum_{\substack{|\gamma|\leq T \\ 0 < \beta-\frac{1}{2} < \delta}}
        (x^{\beta - \frac{1}{2}} + x^{-(\beta-\frac{1}{2})} - 2).
    \end{aligned}
    \end{equation}
    Note that $x^{\beta-1/2}+x^{-(\beta-1/2)}-2=(2\sinh(\frac{1}{2}(\beta-\frac{1}{2})\log x))^2$.
    Note that if $0<\beta-\frac{1}{2}<\delta$, then $0<\frac{1}{2}(\beta-\frac{1}{2})\log x<\frac{1}{2}$.
    Since $y<\sinh(y)<2y$ for $0<y<\frac{1}{2}$, \eqref{eqn:close_to_1/2} is
    \begin{align*}
        \ll (\log x)^2 \sum_{\substack{|\gamma|\leq T \\ 0 < \beta-\frac{1}{2} < \delta}}
        \Big( \beta - \frac{1}{2} \Big)^2 & \ll (\log x)^2 \int_0^{\delta} t N_f(\tfrac{1}{2}+t,T)dt\ll \frac{T(\log x)^2}{\log T}.
    \end{align*}
   The desired result follows.
\end{proof}

\begin{proof}[Proof of Theorem \ref{thm:densityfunction}]
    Let $h\in C^{n+2}(\T^n)$, and let $\bm{\alpha}$ satisfy \eqref{eq:condition_alpha}.
    Let $J\in[1, 100\log(eT)]$. We begin with the expansion
    \begin{align*}
        \sum_{|\gamma| \le T} h(\gamma \bm{\alpha})
        & = \sum_{|\gamma| \le T} \sum_{\bm{m} \in\Z^n} c_{\bm{m}} e^{2 \pi i \gamma (\bm{m} \cdot \bm{\alpha})} \\
        & = N_{\pi}(T) \int_{\mathbb{T}^n} h(\bm{t}) \textup{d}\bm{t}
        + \sum_{|\gamma| \le T}\Big(\sum_{1 \le \|\bm{m}\|_{2} \le J} + \sum_{\|\bm{m}\|_{2} > J}\Big)
        c_{\bm{m}} e^{2 \pi i \gamma (\bm{m} \cdot \bm{\alpha})}.
    \end{align*}
    By Lemma \ref{thm:coefficientbound}, we have
    {\small
    \begin{equation}
    \label{eqn:first_main_term}
        \sum_{|\gamma| \le T} h(\gamma \bm{\alpha})
        -N_{\pi}(T) \int_{\mathbb{T}^n} h(\bm{t}) \textup{d}\bm{t}
        = \sum_{|\gamma| \le T}\sum_{1 \le \|\bm{m}\|_{2} \le J}c_{\bm{m}}e^{2\pi i\gamma(\bm{m}\cdot\bm{\alpha})}+O\Big(\frac{N_{\pi}(T)}{J^{2}}\Big).
    \end{equation}}%
    Write $x_{\bm{m}}=e^{2\pi(\bm{m}\cdot\bm{\alpha})}$.  Since $x_{-\bm{m}}^{i\gamma}=x_{\bm{m}}^{-i\gamma}$
    and $c_{-\bm{m}}=-c_{\bm{m}}$, we find that \eqref{eqn:first_main_term} equals
    \[
         2 \Re \sum_{\substack{1 \le \|\bm{m}\|_{2} \le J \\ \bm{m} \cdot \bm{\alpha} > 0}} c_{\bm{m}} \sum_{0 < \gamma \le T}
         x_{\bm{m}}^{i \gamma} + O \Big( \frac{N_{\pi}(T)}{J^{2}} \Big).
    \]
    Choose $J$ so that $\|\bm{m}\|_2\leq J$ implies $\log x_{\bm{m}}<\frac{c}{3}\frac{\log T}{\log\log T}$ (with $c$ as in Theorem \ref{thm:ZDE}).
    By Lemma \ref{substract_exp} and the above display, \eqref{eqn:first_main_term} equals
    \[
        2 \Re \sum_{\substack{ 1 \le \|\bm{m}\|_{2} \le J \\ \bm{m} \cdot \bm{\alpha} > 0}} c_{\bm{m}} \sum_{|\gamma|\leq T}
        x_{\bm{m}}^{\rho - \frac{1}{2}} + O\Big(\frac{N_{\pi}(T)}{J^{2}} + \frac{T}{\log T} \sum_{1 \le \|\bm{m}\|_{2} \le J}
        \frac{1}{\|\bm{m}\|_2^{n+2} }\Big((\log x_{\bm{m}})^2+\frac{1}{\log T}\Big)\Big).
    \]

    Since $\log x_{\bm{m}}\leq 2\pi\|\bm{m}\|_2 \|\bm{\alpha}\|_2$,
    it follows from our preliminary bound for $J$ that \eqref{eqn:first_main_term} equals
    {\small
    \begin{equation}
    \label{eqn:first_main_two}
        2 \Re \sum_{\substack{ 1 \le \|\bm{m}\|_{2} \le J \\ \bm{m} \cdot \bm{\alpha} > 0}} c_{\bm{m}} \sum_{|\gamma|\leq T}
        x_{\bm{m}}^{\rho - \frac{1}{2}} \\
        + O\Big(\frac{N_{\pi}(T)}{J^{2}} + \frac{T\log\log T}{\log T}\Big).
    \end{equation}}

    We apply Lemma \ref{exp_sum_bound} to conclude that \eqref{eqn:first_main_two} equals
    \begin{equation}
    \label{eqn:second_main_two}
        -\frac{2T}{\pi} \Re \sum_{\substack{1 \le \|\bm{m}\|_{2} \le J \\ \bm{m} \cdot \bm{\alpha} > 0}}
        \frac{c_{\bm{m}} \Lambda_{\pi}(\langle x_{\bm{m}} \rangle)}{\sqrt{x_{\bm{m}}}}
        \cdot \frac{e^{iT \log\frac{x_{\bm{m}} }{ \langle x_{\bm{m}} \rangle}} - 1}
        {iT \log\frac{x_{\bm{m}} }{ \langle x_{\bm{m}} \rangle}} + \mathcal{E},
    \end{equation}
    where $\mathcal{E}$ satisfies (note that $\log x_{\bm{m}}=2\pi(\bm{m}\cdot\bm{\alpha})$)
    \[
        |\mathcal{E}|\ll \sum_{\substack{1 \le \|\bm{m}\|_{2} \le J \\ \bm{m} \cdot \bm{\alpha} > 0}}
        x_{\bm{m}}^{1 + \theta} ( \log(2 x_{\bm{m}}) + \log T )
        + \sum_{\substack{1 \le \|\bm{m}\|_{2} \le J \\ \bm{m} \cdot \bm{\alpha} > 0}} \frac{\log T}{\bm{m}\cdot\bm{\alpha}} + \Big( \frac{N_{\pi}(T)}{J^{2}}
        + \frac{T\log\log T}{\log T} \Big).
    \]
    Our choice of $J$ ensures that $\log(2x_{\bm{m}})\ll \log T$, so it follows from \eqref{eq:condition_alpha} that
    \[
        |\mathcal{E}| \ll \sum_{\substack{1 \le \|\bm{m}\|_{2} \le J \\ \bm{m} \cdot \bm{\alpha} > 0}}
        ( x_{\bm{m}}^{1 + \theta}
        + e^{\|\bm{m}\|_2})\log T + \frac{N_{\pi}(T)}{J^{2}}
        + \frac{T\log\log T}{\log T}.
    \]
    Since $x_{\bm{m}}^{1+\theta}=e^{2\pi(1+\theta)(\bm{m}\cdot\bm{\alpha})}\leq e^{4\pi\|\bm{m}\|_2 \|\bm{\alpha}\|_2}$, it follows that
    \begin{equation}
    \label{eqn:mathcal_E_bound}
        |\mathcal{E}| \ll e^{(n+4\pi\|\bm{\alpha}\|_2)J}\log T + \frac{N_{\pi}(T)}{J^{2}}
        + \frac{T\log\log T}{\log T}.
    \end{equation}
    We choose $J = (\log T)^{2/3}$.  Since $N_{\pi}(T)\ll T\log T$, \eqref{eqn:second_main_two} equals
    \begin{equation}
    \label{second_term2}
        -\frac{2T}{\pi} \Re \sum_{\substack{1 \le \|\bm{m}\|_{2} \le J \\ \bm{m} \cdot \bm{\alpha} > 0}}
        \frac{c_{\bm{m}} \Lambda_{\pi}(\langle x_{\bm{m}} \rangle)}{\sqrt{x_{\bm{m}}}}\cdot
        \frac{e^{iT \log\frac{x_{\bm{m}}}{\langle x_{\bm{m}} \rangle}} - 1}{iT \log\frac{x_{\bm{m}}}{\langle x_{\bm{m}} \rangle}}+ O \Big(\frac{T\log\log T}{\log T}\Big).
   \end{equation}

Observe that if $x_{\bm{m}}\neq\langle x_{\bm{m}}\rangle$, then $|e^{iT\log\frac{x_{\bm{m}}}{\langle x_{\bm{m}}\rangle}}-1|\leq |iT\log\frac{x_{\bm{m}}}{\langle x_{\bm{m}}\rangle}|$.  Also, since $0\leq\theta<\frac{1}{2}$, it follows that $|\Lambda_{\pi}(\langle x_{\bm{m}}\rangle)|\ll \sqrt{x_{\bm{m}}}$.  The proof of Lemma \ref{thm:coefficientbound} ensures that $|c_{\bm{m}}|\ll \|\bm{m}\|_{2}^{-n-2}$, so the sum over $\bm{m}$ converges absolutely.  By the decay of $|c_{\bm{m}}|$, our choice of $J$, and \eqref{eqn:mathcal_E_bound},  \eqref{second_term2} equals
    \begin{multline}
    \label{eqn:second_main_2}
       -\frac{2T}{\pi} \Re \Big(\sum_{\substack{ \bm{m} \cdot \bm{\alpha} > 0 \\ x_{\bm{m}}=\langle x_{\bm{m}}\rangle}}
       \frac{c_{\bm{m}} \Lambda_{\pi}(x_{\bm{m}})}{\sqrt{x_{\bm{m}}}}+\sum_{\substack{ \bm{m} \cdot \bm{\alpha} > 0 \\ x_{\bm{m}}\neq \langle x_{\bm{m}}\rangle}}
        \frac{c_{\bm{m}} \Lambda_{\pi}(\langle x_{\bm{m}}\rangle)}{\sqrt{x_{\bm{m}}}} \cdot \frac{e^{iT \log\frac{x_{\bm{m}} }{ \langle x_{\bm{m}} \rangle}} - 1}
        {iT \log\frac{x_{\bm{m}} }{ \langle x_{\bm{m}} \rangle}}\Big)\\
        +O \Big(\frac{T\log\log T}{\log T}\Big).
    \end{multline}
In particular, each sum over $\bm{m}$ converges absolutely.

To handle the sum over $\bm{m}$ such that $x_{\bm{m}}\neq\langle x_{\bm{m}}\rangle$, we note (by absolute convergence) that for all $\varepsilon>0$, there exists $M_{\varepsilon}=M_{\varepsilon}(\alpha,h)>0$ such that
\[
    \Big| \sum_{\substack{ \|\bm{m}\|_{2} >M_{\epsilon} \\ \bm{m} \cdot \bm{\alpha} > 0 \\ x_{\bm{m}}\neq \langle x_{\bm{m}}\rangle}}
    \frac{c_{\bm{m}} \Lambda_{\pi}(\langle x_{\bm{m}} \rangle)}{\sqrt{x_{\bm{m}}}}
    \cdot \frac{e^{iT \log\frac{x_{\bm{m}} }{ \langle x_{\bm{m}} \rangle}} - 1}
    {iT \log\frac{x_{\bm{m}} }{ \langle x_{\bm{m}} \rangle}}\Big|<\epsilon.
\]
Consequently, we have
\begin{multline*}
    \Big|-\frac{2T}{\pi} \Re \sum_{\substack{\|\bm{m}\|_{2} \geq 1 \\ \bm{m} \cdot \bm{\alpha} > 0 \\ x_{\bm{m}}\neq \langle x_{\bm{m}}\rangle}}
        \frac{c_{\bm{m}} \Lambda_{\pi}(\langle x_{\bm{m}} \rangle)}{\sqrt{x_{\bm{m}}}}
        \cdot \frac{e^{iT \log\frac{x_{\bm{m}} }{ \langle x_{\bm{m}} \rangle}} - 1}
        {iT \log\frac{x_{\bm{m}} }{ \langle x_{\bm{m}} \rangle}}\Big|\\
\leq \Big|\Re \sum_{\substack{ \|\bm{m}\|_{2} \leq M_{\epsilon} \\ \bm{m} \cdot \bm{\alpha} > 0 \\ x_{\bm{m}}\neq \langle x_{\bm{m}}\rangle}}
        \frac{c_{\bm{m}} \Lambda_{\pi}(\langle x_{\bm{m}} \rangle)}{\sqrt{x_{\bm{m}}}}
        \cdot \frac{e^{iT \log\frac{x_{\bm{m}} }{ \langle x_{\bm{m}} \rangle}} - 1}
        {\log\frac{x_{\bm{m}} }{ \langle x_{\bm{m}} \rangle}}\Big|+\epsilon T.
\end{multline*}
As we let $\epsilon\to 0$ sufficiently slowly, we conclude that \eqref{eqn:second_main_two} equals $o(T)$, as desired.

For the sum over $\bm{m}$ such that $x_{\bm{m}} = \langle x_{\bm{m}} \rangle$,
which means $x_{\bm{m}} \in \Z$, the terms which are not prime powers will vanish due to the presence of the von Mangoldt function.
For the other terms which are prime powers, we have $\bm{m} \cdot \bm{\alpha} = \frac{k \log p}{2 \pi}$ for some $k \in \N$
by the definition of $x_{\bm{m}}$. This will only happen when $\bm{m}$ is a multiple of $q_j \bm{b}_j$ for some $j \in \{1, \ldots, r\}$
due to our choice of the vector $\bm{\alpha}$ in \eqref{eqn:matrix_condition}, so
\begin{align*}
    -\frac{2}{\pi} \Re \sum_{\substack{\bm{m} \cdot \bm{\alpha} > 0 \\ x_{\bm{m}}=\langle x_{\bm{m}}\rangle}}
    \frac{c_{\bm{m}} \Lambda_{\pi}(x_{\bm{m}})}{\sqrt{x_{\bm{m}}}}= -\frac{2}{\pi} \Re \sum_{j = 1}^{r} \sum_{l = 1}^{\infty} \frac{\Lambda_\pi(p_j^{a_j l})}{p_j^{a_j l/2}} c_{l q_j \bm{b}_j}= \int_{\mathbb{T}^n} h(\bm{t}) g_{\pi, \bm{\alpha}}(\bm{t})\textup{d}\bm{t}.
\end{align*}
The last equation holds because of \eqref{eq:densityfunction} and \eqref{eq:cm}.
\end{proof}

\section{Proof of Corollary \ref{cor:FSZ}}

Let $\mathbb{B}\subseteq\T^n$ be a product of $n$ subintervals of $\T$ for which $|\int_{\mathbb{B}}g_{f,\bm{\alpha}}(\bm{t})\textup{d}\bm{t}|$ attains its maximum.  For $\epsilon>0$, let $\varphi_{\epsilon}:\T^n\to\mathbb{R}$ satisfy the following conditions:
\begin{enumerate}[(i)]
    \item $\varphi_{\epsilon}$ is nonnegative and infinitely differentiable,
    \item $\varphi_{\epsilon}$ is supported on a compact subset of $U_{\epsilon}:=\{\bm{t}\in \T^n\colon \|\bm{t}\|_2<\epsilon\}$, and
    \item $\int_{\T^n}\varphi_{\epsilon}(\bm{t})\textup{d}\bm{t}=1$.
\end{enumerate}
Let $\mathbf{1}_{\mathbb{B}}$ be the indicator function of $\mathbb{B}$, and define
$h_{\epsilon}(\bm{t})=\int_{\T^n}\varphi_{\epsilon}(\bm{t})\mathbf{1}_{\mathbb{B}}(\bm{x}-\bm{t})\textup{d}\bm{t}$.
Then $h_{\epsilon}$ is infinitely differentiable, and thus Theorem \ref{thm:densityfunction} holds with $r$ arbitrarily large for $h=h_{\epsilon}$.
Consequently, for any fixed $r\geq n+2$, we have
\[
\int_{\T^n}h_{\epsilon}(\bm{y})\Big(\sum_{\substack{|\gamma|\leq T \\ \{\gamma\bm{\alpha}\}\in\mathbb{B}+\bm{y}}}1-\mathrm{vol}(\mathbb{B})N_{\pi}(T)\Big)d\bm{y}=T\int_{U_{\epsilon}}h_{\epsilon}(\bm{y})\int_{\mathbb{B}+\bm{y}}g_{f,\bm{\alpha}}(\bm{x})d\bm{x}d\bm{y} + o(T).
\]
It follows from our definition of $g_{f,\bm{\alpha}}(\bm{t})$ in \eqref{eq:densityfunction} that $g_{f,\bm{\alpha}}(\bm{t})\ll 1$, hence
\[
    \Big|\int_{\mathbb{B}+\bm{y}}g_{f,\bm{\alpha}}(\bm{x})d\bm{x}-\int_{\mathbb{B}}g_{f,\bm{\alpha}}(\bm{x})d\bm{x}\Big|\ll \epsilon
\]
for all $\bm{y}\in U_{\epsilon}$.  Thus, we have
\[
\int_{\T^n}h_{\epsilon}(\bm{y})\Big(\sum_{\substack{|\gamma|\leq T \\ \{\gamma\bm{\alpha}\}\in\mathbb{B}+\bm{y}}}1-\mathrm{vol}(\mathbb{B})N_{\pi}(T)\Big)d\bm{y}= T\int_{\mathbb{B}}g_{f,\bm{\alpha}}(\bm{t})\textup{d}\bm{t}+O(\epsilon T) + o(T).
\]
By the mean value theorem, there exists $\bm{y}\in U_{\epsilon}$ such that
\[
    \Big|\sum_{\substack{|\gamma|\leq T \\ \{\gamma\bm{\alpha}\}\in\mathbb{B}+\bm{y}}}1-\mathrm{vol}(\mathbb{B})N_{\pi}(T)\Big|
    \geq T\Big|\int_{\mathbb{B}}g_{f,\bm{\alpha}}(\bm{t})\textup{d}\bm{t}\Big| + O(\epsilon T) + o(T).
\]
The proof follows once we let $\epsilon \to 0$ sufficiently slowly as a function of $T$.

\section{Proofs of Corollaries 1.5-1.8}
\subsection{An estimate for the density function}

We begin with a useful estimate for the density function $g_{f,\bm{\alpha}}$ associated to a holomorphic cuspidal newform $f\in S_k^{\mathrm{new}}(\Gamma_0(q))$ as in Section \ref{sec:holomorphic_newforms}.
\begin{lemma}
\label{lem:truncation}
    Let $f \in S_k (\Gamma_0(q))$ be a newform and let $\alpha = \frac{a\log p}{2 \pi q}$.    Then we have
    \[
        \Big|g_{f, \alpha}(t) + \frac{2}{\pi}
        \frac{\Lambda_{\pi} (p^{a})}{p^{a/2}} \cos (2 \pi q t) \Big|
        \leq  \frac{4 \log p}{\pi p^a (1 - p^{-a/2})}.
    \]
\end{lemma}
\begin{proof}
    In this case, Deligne's bound implies that
    \[
    \Big|g_{f,\alpha}(t)+\frac{2}{\pi}\frac{\Lambda_{\pi}(p^a)}{p^{a/2}}\cos(2\pi q t)\Big| = \Big|\frac{2}{\pi} \sum_{\ell = 2}^{\infty} \frac{\Lambda_{\pi}(p^{a \ell})}{p^{\frac{a \ell }{2}}} \cos ( 2 \pi q \ell t)\Big| \le  \sum_{\ell=2}^{\infty} \frac{ 2 \log p   }{p^{a \ell /2}}
        = \frac{2  \log p}{p^a (1 - p^{-a /2}) }.
    \]
\end{proof}

\subsection{Proof of Corollary \ref{cor:equalconductors}}
First we prove a simple criterion for $f_1$ winning the $(\frac{\log p}{2\pi}, h)$-race. From Lemma \ref{lem:truncation}, we obtain
\[
    \int_\mathbb{T} h(t) ( g_{f_1, \alpha}(t) - g_{f_2, \alpha}(t) ) dt > \frac{(\lambda_{f_2}(p) - \lambda_{f_1}(p)) 2 \log p}{\pi \sqrt{p}} \int_\mathbb{T} h(t) \cos(2 \pi t)dt  - \frac{\int_\mathbb{T} h(t) dt \cdot 8 \log p}{\pi p (1 - p^{-1/2}) }.
\]
Consequently, the inequality
\begin{equation}
\label{eq:estimate1}
(\lambda_{f_2}(p) - \lambda_{f_1}(p))\int_\mathbb{T} h(t) \cos(2 \pi t)dt > \frac{4 \int_{\mathbb{T}} h(t) dt}{\sqrt{p}(1-p^{-1/2})}.
\end{equation}
implies that
\[
    \int_\mathbb{T} h(t) ( g_{f_1, \alpha}(t) - g_{f_2, \alpha}(t) ) dt > 0,
\]
which by Theorem \ref{thm:densityfunction} tells us that $f_1$ wins the $(\frac{\log p}{2 \pi}, h)$-race.

Throughout the proof, we let $ k_h := \int_\mathbb{T} h(t) \cos(2 \pi t)dt$.  The number of $p$ for which $f_1$ wins the race is equal to $T_1 - T_2 + T_3$, where
\begin{align*}
    T_1 &:=  \# \{ p \le X:  k_h(\lambda_{f_2} (p) - \lambda_{f_1}(p)) > 0 \}, \\
    T_2 &:= \# \{ p \le X:  f_1 \mbox{ loses the $(\alpha,h)$-race, and } k_h( \lambda_{f_2} (p) - \lambda_{f_1}(p)) > 0 \},\\
    T_3 &:= \# \{ p \le X: f_1 \mbox{ wins the $(\alpha,h)$-race, and } k_h( \lambda_{f_2} (p) - \lambda_{f_1}(p)) \le 0 \}.
\end{align*}
From the symmetry of $\mu_{ST,2}$, we see that $\nu(\{ (x,y) \in [-2,2]^2: x-y > 0\}) = \frac{1}{2}$. By Corollary \ref{cor:STlimit} with $\mathcal{I} = (0,4)$, we have
\[
    T_1 = \frac{1}{2} \pi(X) + O\Big( \pi(X) \frac{\sqrt{\log\log\log X}}{(\log\log X)^{1/4}} \Big).
\]

We have to show that $T_2$ and $T_3$ are both $O( \pi(X) \frac{\sqrt{\log\log\log X}}{(\log\log X)^{1/4}})$. Note that
\[
T_2 = \# \{ \sqrt{X} \le p \le X: f_1 \mbox{ loses the $(\alpha, h)$-race, and } k_h(\lambda_{f_2}(p) - \lambda_{f_1}(p)) > 0 \} + O(\pi(\sqrt{X})).
\]
If $p \ge \sqrt{X}$ and $f_1$ loses the $(\alpha, h)$-race and $k_h(\lambda_{f_2}(p) - \lambda_{f_1}(p)) > 0$,
then by \eqref{eq:estimate1}, we have
\[
    k_h(\lambda_{f_2}(p) - \lambda_{f_1}(p)) \in \Big(0, \frac{4 \int_{\mathbb{T}} h(t) dt }{\sqrt{p}(1-p^{-1/2}) } \Big]
    \subseteq \Big( 0, \frac{4 \int_{\mathbb{T}} h(t) dt}{X^{\frac{1}{4}}(1-X^{-1/2})} \Big].
\]
Denoting by $J_X$ the rightmost interval in the preceding display, we have
\begin{align*}
    \# & \{ p \le X: f_1 \mbox{ loses the $(\alpha, h)$-race, and } k_h(\lambda_{f_2}(p) - \lambda_{f_1}(p)) > 0 \} \\
    & \le \# \{ p \le X: \lambda_{f_2}(p) - \lambda_{f_1}(p) \in J_X \} + O(\pi(\sqrt{X})).
\end{align*}
By Corollary \ref{cor:STlimit}, this is at most
\begin{multline*}
   \# \{ p \le X: f_1 \mbox{ loses the $(\alpha, h)$-race, and }k_h( \lambda_{f_2} (p) - \lambda_{f_1}(p)) > 0 \} \\
   \leq \nu \Big(  |k_h|^{-1} J_X  \Big) \pi(X) +O\Big(\pi(X) \frac{\sqrt{\log\log\log X}}{(\log\log X)^{1/4}}\Big).
\end{multline*}
Since $\nu(|k_h|^{-1} J_X) = O( | J_X | ) = O(X^{-1/4})$, it follows that
\[
    \# \{p \le X: f_1 \mbox{ loses the $(\alpha, h)$-race, and } k_h(\lambda_{f_2}(p) - \lambda_{f_1}(p)) > 0 \} = O\Big(\pi(X) \frac{\sqrt{\log\log\log X}}{(\log\log X)^{1/4}}\Big).
\]

If the conditions for $T_3$ are true, then
\[
    k_h( \lambda_{f_2}(p) - \lambda_{f_1}(p) )\in \Big[- \frac{4 \int_{\mathbb{T}} h(t) dt }{\sqrt{p}(1-p^{-1/2}) }, 0 \Big].
\]
Therefore, $T_3\ll \pi(X) \frac{\sqrt{\log\log\log X}}{(\log\log X)^{1/4}}$ by the argument used to bound $T_2$.

The result follows from the estimate shown for $T_1$ and the bounds for $T_2$ and $T_3$.

\subsection{Proof of Corollary \ref{thm:STracedist}}
\begin{proof}
    From Theorem \ref{thm:densityfunction}, Lemma \ref{lem:truncation}, and (\ref{eq:zerocount}) we have
    \begin{equation}\label{eq:firstexp}
        H\Big(f_1, f_2, h, \frac{\log p}{2 \pi}\Big)
        = \frac{2 k_h}{\pi} \cdot \frac{\log p \cdot (\lambda_{f_2}(p) - \lambda_{f_1}(p))}{p^{1 /2}}
        + O_{h} \Big( \frac{\log p}{p} \Big).
    \end{equation}
    Consider the statements
    \begin{equation}\label{eq:coefficientcondition}
       \frac{2 k_h}{\pi} (\lambda_{f_2}(p) - \lambda_{f_1}(p)) \in \mathcal{I}
    \end{equation}
    and
    \begin{equation} \label{eq:sumcondition}
       \frac{\sqrt{X}}{\log{X}} H(f_1, f_2, h, \frac{\log p}{2 \pi}) \in \mathcal{I}.
    \end{equation}
    Defining
    \begin{align*}
        T_1&:=\# \{ p \in [(1 - \epsilon_X)X, X]: \mbox{ \eqref{eq:coefficientcondition} holds} \},\\
        T_2&:=\# \{ p \in [(1 - \epsilon_X)X, X]: \mbox{ \eqref{eq:coefficientcondition} holds and \eqref{eq:sumcondition} fails} \},\\
        T_3&:=\# \{ p \in [(1 - \epsilon_X)X, X]: \mbox{ \eqref{eq:coefficientcondition} fails and \eqref{eq:sumcondition} holds} \},
    \end{align*}
we have
    \begin{align*}
        \# \Big\{ p \in [(1 - \epsilon_X)X, X]:  \frac{\sqrt{X}}{\log{X}} H\Big(f_1,  f_2, h, \frac{\log{p}}{2 \pi}\Big) \in \mathcal{I} \Big\} = T_1 - T_2 + T_3.
    \end{align*}
    By Corollary \ref{cor:STlimit} we have:
    \begin{align*}
        T_1&= \# \{ p \le X: \mbox{ \eqref{eq:coefficientcondition} holds} \} - \# \{ p \le (1-\epsilon_X)X: \mbox{ \eqref{eq:coefficientcondition} holds} \} \\
        &= \nu \Big(\frac{\pi}{2k_h} \mathcal{I}\Big) \pi(X) + O( \pi(X) \epsilon_X^2 ) \\
        &- \nu \Big(\frac{\pi}{2k_h} \mathcal{I} \Big) \pi((1-\epsilon_X)X) + O \Big( \pi(X - \epsilon_X X) \epsilon_X^2 \Big) \\
        &= \nu \Big(\frac{\pi}{2k_h} \mathcal{I} \Big) \cdot \Big(\pi(X) - \pi((1 - \epsilon_X) X) \Big) + O\Big( \pi(X) \epsilon_X^2 \Big).
    \end{align*}

    We proceed to show that $T_2$ and $T_3$ are $O\Big( \epsilon_X^2\pi(X)  \Big)$ as $X \to \infty$. We first examine $T_2$.
    Set $\mathcal{I} = (\delta_1, \delta_2)$. If the condition in \eqref{eq:sumcondition} is false, then
    \[
        \frac{\sqrt{X}}{\log{X}} H\Big(f_1,f_2, h, \frac{\log p}{2 \pi}\Big) \notin [\delta_1,\delta_2].
    \]
    Applying \eqref{eq:firstexp}, we deduce that
    \[
        \frac{\sqrt{X}}{\log{X}}  \frac{2k_h}{\pi} \cdot \frac{\log p \cdot (\lambda_{f_2}(p) - \lambda_{f_1}(p))}{p^{1 /2}} \notin \Big[\delta_1 + \frac{C\log{p}}{p},\delta_2 - \frac{C \log{p}}{p}\Big],
    \]
    where $C$ is an implied constant in \eqref{eq:firstexp}. Then \eqref{eq:coefficientcondition} gives us
    \begin{equation}\label{eq:restriction}
        \frac{2 k_h}{\pi} (\lambda_{f_2}(p) - \lambda_{f_1}(p))
        \in \Big( \delta_1, \delta_1 \Big(\frac{p^{1/2} \log X}{X^{1/2} \log{p}}\Big) + \frac{C}{\sqrt{p}} \Big) \\
        \cup \Big( \delta_2 \Big(\frac{p^{1/2} \log X}{X^{1/2} \log{p}}\Big)  - \frac{C}{\sqrt{p}}, \delta_2 \Big).
    \end{equation}
     Since $p \in [(1-\epsilon_X)X, X]$, it follows that $(\lambda_{f_2}(p) - \lambda_{f_1} (p))  \in I_X$, where $I_X$ is
    \begin{align*}  
        \frac{\pi}{2k_h} \Big(\delta_1, \frac{\delta_1 \log X}{\log{X(1-\epsilon_X)}} + \frac{C}{\sqrt{X(1-\epsilon_X)}}\Big) \cup \frac{\pi}{2k_h} \Big( \delta_2 \sqrt{1-\epsilon_X} - \frac{C}{\sqrt{X(1-\epsilon_X)}}, \delta_2 \Big).
    \end{align*}
    By Corollary \ref{cor:STlimit}, we have
    \begin{align*}
        & \# \{ (1 - \epsilon_X)X \le p \le X: (\lambda_{f_1}(p) - \lambda_{f_2}(p)) \in I_X \} \\
        &=  \nu(I_X) (\pi(X) - \pi((1 - \epsilon_X)X)) + O\Big(\pi(X) \epsilon_X^2 \Big).
    \end{align*}
    
    From the prime number theorem, we obtain
    $$
    \pi(X) - \pi((1 - \epsilon_X)X) \sim \epsilon_X \pi(X) . 
    $$
    
   Combining this with the fact that $\nu(I_X) = O(\epsilon_X)$, we conclude the following:
    $$
     \# \{ (1 - \epsilon_X)X \le p \le X: (\lambda_{f_1}(p) - \lambda_{f_2}(p)) \in I_X \}  = O( \epsilon_X^2 \pi(X) ).
     $$
     
     Therefore, $T_2 = O( \epsilon_X^2 \pi(X))$. A very similar argument can be used to bound $T_3$. More specifically, if \eqref{eq:sumcondition} holds, then we have
    \[
        \frac{\sqrt{X}}{\log{X}}  \frac{2k_h}{\pi} \cdot \frac{\log p \cdot (\lambda_{f_2}(p) - \lambda_{f_1}(p))}{p^{1 /2}} \in ( \delta_1 - \frac{C\log{p}}{p},\delta_2 + \frac{C \log{p}}{p}),
    \]
    If (\ref{eq:coefficientcondition}) fails and \eqref{eq:sumcondition} holds, then, much like \eqref{eq:restriction}, we obtain
    \[
        \frac{2 k_h}{\pi} (\lambda_{f_2}(p) - \lambda_{f_1}(p))
        \in \Big( \delta_1 \Big(\frac{p^{1/2} \log X}{X^{1/2} \log{p}}\Big) - \frac{C}{\sqrt{p}}, \delta_1 \Big) \cup \Big(\delta_2, \delta_2 \Big(\frac{p^{1/2} \log X}{X^{1/2} \log{p}}\Big) + \frac{C}{\sqrt{p}} \Big).
    \]
    By Corollary \ref{cor:STlimit}, the number of such $p \in ((1-\epsilon_X)X, X)$ is $O\Big( \epsilon_X^2\pi(X)  \Big)$.
\end{proof}

\subsection{Proof of Corollary \ref{thm:distinctconductors}}
By Theorem \ref{thm:densityfunction}, if $f_2$ wins the $(\alpha, h)$-race, where $\alpha = \frac{\log p}{2\pi}$, then we must have
\[
    0 < \frac{\log(q_1 / q_2)}{ \pi} \int_\mathbb{T} h(t) dt < \int_\mathbb{T} h(t) ( g_{f_2, \alpha}(t) - g_{f_1, \alpha}(t)) dt.
\]
By Lemma \ref{lem:truncation}, we have
\[
    \Big| \int_\mathbb{T} (g_{f_2, \alpha}(t) - g_{f_1,\alpha}(t)) h(t) dt \Big| \leq \int_\mathbb{T} h(t) dt \cdot \frac{2 \log{p}}{\pi}
     \Big( p^{ - \frac{1}{2}} + \frac{1}{(p^{ \frac{1}{2} } - 1)
    p^{\frac{1}{2}}} \Big).
\]
It follows that $\log(q_1/q_2) \le   2 p^{-1/2}( 1 + (p^{ \frac{1}{2} } - 1 )^{-1})\log{p}$.  The left hand side is independent of $\alpha$ and positive, while the right hand side tends to zero as $p$ grows.  Thus, this inequality holds for only finitely many primes $p$.

\subsection{Proof of Corollary \ref{thm:exceptional}}
Fix $t_0 \in [0,1)$, and $k_1, k_2 \in \mathbb{Z}$.  By Theorem \ref{thm:densityfunction} and the same reasoning as in \cite[Theorem 1.2]{LZ}, that $f_1$ wins the local $(\alpha, t_0)$-race against $f_2$ if
\[
    g_{f_1, \alpha}(t_0) + \frac{\log{q_1}}{2 \pi} > g_{f_2, \alpha} (t_0) + \frac{\log q_2}{2\pi},
\]
or, equivalently, if $\frac{1}{2\pi} \log(q_1/q_2) > g_{f_2, \alpha}(t_0) - g_{f_1, \alpha}(t_0)$.

We first assume $t_0 \neq \frac{1}{4}, \frac{3 }{4}$.  For $\alpha = \frac{\log p}{2 \pi}$ and $q_2 \in (q_1, q_1 + \sqrt{q_1})$, by Lemma~\ref{lem:truncation}, the following is sufficient to guarantee that $f_1$ wins the $(\alpha,t_0)$ race:
\begin{equation}\label{eq:STgoal}
    \frac{\sqrt{p}}{2 \log p} \log(1 + q_1^{-1/2}) < \cos(2 \pi t_0) (\lambda_{f_2}(p) - \lambda_{f_1}(p)) - \frac{4}{p^{1/2} (1 - p^{-1/2})}.
\end{equation}
This inequality is automatically true if both
\begin{equation}\label{eq:constraint1}
    (\lambda_{f_1}(p) - \lambda_{f_2}(p)) \cos (2 \pi t_0) \ge \frac{|\cos(2 \pi t_0)|}{2}
\end{equation}
and
\begin{equation}\label{eq:constraint2}
     \frac{|\cos(2 \pi t_0)|}{2} > \frac{p^{1/2}}{2 \log p} \log(1 + q_1^{-1/2})+ \frac{4}{p^{1/2} (1 - p^{-1/2})}
\end{equation}
are satisfied.

Now set $X =  q_1^{1/4}  $ and $Y = q_1^{1/6} $. For $p \in [Y,X)$, if $q_1$ is sufficiently large, then (\ref{eq:constraint2}) is satisfied.
Choose $I_1, I_2 \subseteq [-2, 2]$ such that $(\lambda_{f_1}(p) ,\lambda_{f_2}(p)) \in I_1 \times I_2$ implies \eqref{eq:constraint1}. Suppose $f_1 \in S_{k_1}^{new}(\Gamma_0(q_1)), f_2 \in S_{k_2}^{new} (\Gamma_0(q_2))$ are non-CM newforms, where $q_2 \in [q_1, q_1 + \sqrt{q_1}]$, and assume, as in the statement of the Corollary, that $f_2 \neq f_1 \otimes \chi$ for any primitive Dirichlet character $\chi$. Let
$$
\pi_{f_1, f_2, I_1, I_2}(X) := \# \{ p \le X: (\lambda_{f_1}(p) , \lambda_{f_2} (p)) \in I_1 \times I_2 \}.
$$
Then by Theorem \ref{thm:effectiveST} we have
\begin{multline*}
    \pi_{f_1, f_2, I_1, I_2}(X) - \pi_{f_1, f_2, I_1, I_2}(Y) \ge \mu_{\mathrm{ST}}(I_1) \mu_{\mathrm{ST}}(I_2) ( \pi(X) - \pi(Y)) \\
    - c \pi(X) \frac{\log \log \log (k_1 k_2 q_1^{\frac{1}{4} +\frac{1}{2} + 1})}{(\log \log q_1^{1/4})^{1/2}}  - c \pi(Y) \frac{\log \log \log (k_1 k_2 q_1^{\frac{1}{6} +\frac{1}{2} + 1})}{(\log \log q_1^{1/6})^{1/2}}.
\end{multline*}
So if $q_1$ is sufficiently large, then $\pi_{f_1, f_2, I_1, I_2}(X) > \pi_{f_1, f_2, I_1, I_2}(Y)$. So there exists $p$ between $Y$ and $X$ such that (\ref{eq:STgoal}) is satisfied and
therefore $f_1$ wins the local $(t_0, \frac{\log p}{2\pi})$ race.

Finally, if $t_0 = \frac{1}{4}, \frac{3}{4}$, then instead of \eqref{eq:STgoal}, we wish to find $p$ such that
\[   
    \frac{p}{2 \log p} \log(1 + q_1^{-1/2}) < (\lambda_{f_1}(p)^2 - \lambda_{f_2}(p)^2) - \frac{4}{p^{3/2} (1 - p^{-1/2})}.
\]
We obtain this the same way as the first case.

\section{Example}
We conclude with a numerical example to illustrate \eqref{eqn:conj_GL2}.
For our example, we consider the $L$-function $L(s,\Delta)$ associated to the discriminant modular form
\[
\Delta(z) = e^{2\pi i z}\prod_{n=1}^{\infty}(1-e^{2\pi i n z})^{24} = \sum_{n=1}^{\infty}\tau(n)e^{2\pi inz}\in S_{12}^{\mathrm{new}}(\Gamma_0(1)),
\]
where $\tau(n)$ denotes the Ramanujan tau function.  We use Rubinstein's \texttt{lcalc} package \cite{MR} to calculate the $2\cdot 10^5$ nontrivial zeros $L(s, \Delta)$ up to height $T = 74920.77$.

Let $M$ and $\bm{\alpha}$ satisfy the following relation for \eqref{eqn:matrix_condition},
\[
    M \bm{\alpha}^{\intercal}
    =
    \begin{pmatrix}
        1 & 1 \\
        1 & 2
    \end{pmatrix}
    \begin{pmatrix}
        \alpha_1 \\
        \alpha_2
    \end{pmatrix}
    =
    \frac{1}{2\pi}\begin{pmatrix}
    \log 2 \\
    \log 3
    \end{pmatrix},
\]
so that \eqref{eq:densityfunction} will define our density function $g_{\Delta, \bm{\alpha}}(x,y)$.
We graph $g_{\Delta, \bm{\alpha}}(x,y)$ in Figure \ref{fig:theory} below.
Next, we partition the unit square $[0, 1) \times [0, 1)$ as
\[
    [0,1)\times[0,1) = \bigcup_{a=0}^{29}\bigcup_{b=0}^{29}S_{a,b},
    \qquad S_{a,b}:=\Big[\frac{a}{30},\frac{a+1}{30}\Big)\times\Big[\frac{b}{30},\frac{b+1}{30}\Big).
\]
Given $(x,y)\in[0,1)\times[0,1)$, there exists a unique pair of integers $a$ and $b$ with $0\leq a,b\leq 29$ such that $(x,y)\in S_{a,b}$.
Denoting this unique square as $S(x,y)$, we define
\[
    \tilde{g}_{\Delta, \bm{\alpha}}(x, y) := \#\{\rho = \beta + i \gamma: L(\rho, \Delta) = 0 \textup{ and }
    (\{\alpha_1 \gamma\}, \{\alpha_2 \gamma\}) \in S(x,y) \}.
\]
This gives us a discretized approximation to $g_{\Delta,\bm{\alpha}}(x,y)$, which we plot in Figure \ref{fig:data} below.

\begin{figure}[!ht]
    \centering
    \begin{subfigure}[b]{0.45\linewidth}
        \includegraphics[width=\linewidth]{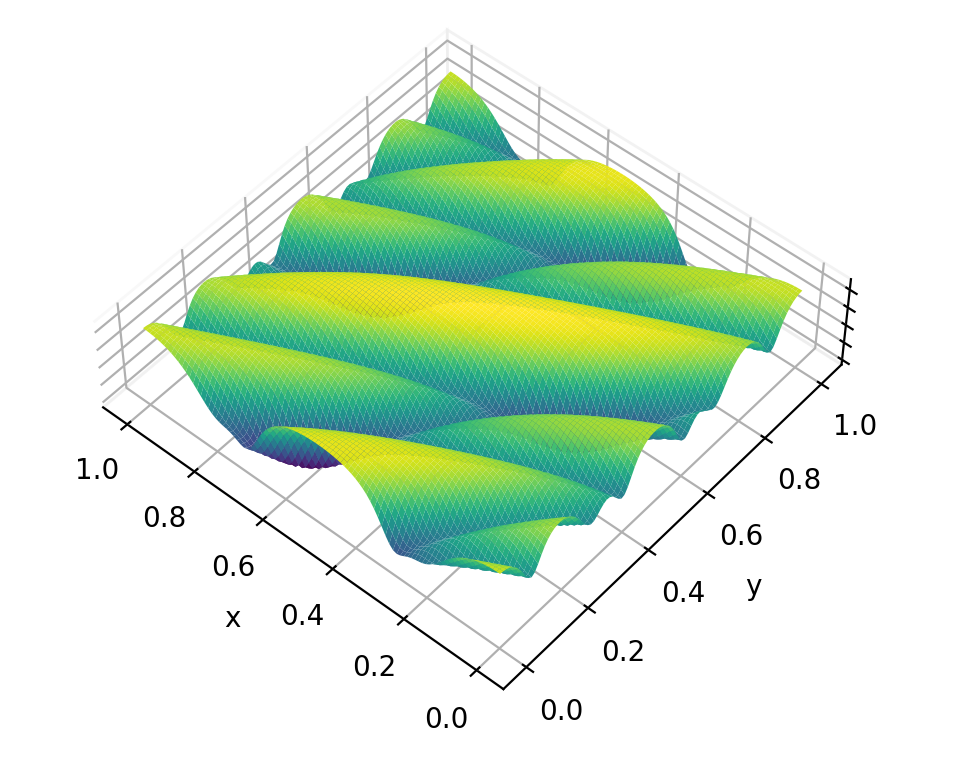}
        \caption{$g_{\Delta, \bm{\alpha}}(x,y)$.}
        \label{fig:theory}
    \end{subfigure}
    \begin{subfigure}[b]{0.45\linewidth}
        \includegraphics[width=\linewidth]{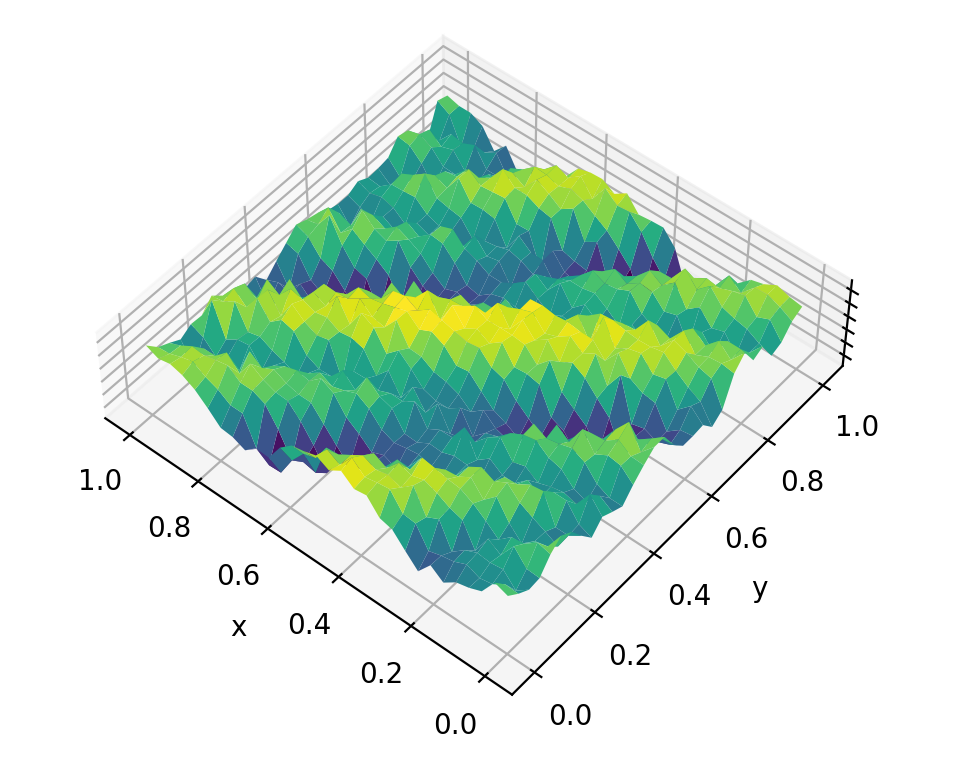}
        \caption{$\tilde{g}_{\Delta, \bm{\alpha}}(x,y)$.}
        \label{fig:data}
    \end{subfigure}
    \caption{Example.}
\end{figure}

\bibliographystyle{amsalpha}
\bibliography{ZeroRaces}

\end{document}